\documentclass[10pt]{amsart}
\usepackage{latexsym,amssymb,enumerate,graphicx,psfrag,mathrsfs}

\begin{document}

\theoremstyle{plain}
  \newtheorem{theorem}{Theorem}[section]
  \newtheorem{proposition}[theorem]{Proposition}
  \newtheorem{lemma}[theorem]{Lemma}
  \newtheorem{corollary}[theorem]{Corollary}
  \newtheorem{conjecture}[theorem]{Conjecture}
\theoremstyle{definition}
  \newtheorem{definition}[theorem]{Definition}
  \newtheorem{example}[theorem]{Example}
  \newtheorem{observation}[theorem]{Observation}
  \newtheorem{observations}[theorem]{Observations}
  \newtheorem{question}[theorem]{Question}
 \theoremstyle{remark}
  \newtheorem{remark}[theorem]{Remark}

\numberwithin{equation}{section}

\def\ZZ{{\mathbb Z}}
\def\QQ{{\mathbb Q}}
\def\CC{{\mathbb C}}
\def\NN{{\mathbb N}}

\def\Kbar{{\overline{K}}}

\def\Y{{\mathbf Y}}
\def\SY{{S\mathbf Y}}
\def\YF{{\mathbf Y F}}
\def\il{{\text{Smith-eigenvalue}}}
\def\pd{{\frac{\partial}{\partial p_1}}}
\def\K{{\mathbb C}}
\def\E{{\mathcal E}}
\def\A{{\mathscr A}}
\def\P{{\mathbf P}}
\def\Q{{\mathbf Q}}

\def\symm{\mathfrak{S}}

\def\coker{\mathrm{coker}}
\def\im{\mathrm{im}}
\def\diag{\mathrm{diag}}
\def\rank{\mathrm{rank}}
\def\Res{\mathrm{Res}}
\def\Ind{\mathrm{Ind}}
\def\triv{{\mathbf{1}}}
\def\Irr{\mathrm{Irr}}
\def\ones{{\mathrm{ones}}}

\def\Class{{\mathbf{Cl}}}

\title{Differential posets and Smith normal forms}
\author{Alexander Miller}
\email{mill1966@math.umn.edu}
\author{Victor Reiner}
\email{reiner@math.umn.edu}
\address{ School of Mathematics\\
University of Minnesota\\
Minneapolis, MN 55455}

\begin{abstract}  
We conjecture a strong property for the up and down maps $U$ and $D$
in an $r$-differential poset: $DU+tI$ and $UD+tI$ have Smith normal forms over $\ZZ[t]$.   
In particular, this would determine the integral structure of the maps $U,D,UD,DU$, 
including their ranks in any characteristic.

As evidence, we prove the conjecture for the
Young-Fibonacci lattice $\YF$ studied by Okada and its $r$-differential generalizations
$Z(r)$, as well as verifying many of its consequences for Young's lattice $\Y$
and the $r$-differential Cartesian products $Y^r$.
\end{abstract}

\thanks{
The authors thank Janvier Nzeutchap, John Shareshian, and Richard Stanley for helpful conversations,
and they particularly thank Patrick Byrnes for helpful computations as well as conversations.
}


\maketitle

%\tableofcontents

%%%%%%%%%%%%%%%%%%%%%%%%%%%%%%%%%%%%%%%%%%%%%%%%%%%%
\section{Introduction and the main conjecture}
\label{sec:intro}
%%%%%%%%%%%%%%%%%%%%%%%%%%%%%%%%%%%%%%%%%%%%%%%%%%%%

The focus of this paper is Conjecture~\ref{conj:parametrized-Smith} below.
We provide here some background and motivation before stating it.

For a positive integer $r$, an {\it $r$-differential poset} is a graded poset $P$ with a minimum element,
having all intervals and all rank cardinalities finite, satisfying two crucial extra conditions:  
\begin{enumerate}
\item[(D1)] If an element of $P$ covers $m$ others, then it
will be covered by $m+r$ others.
\item[(D2)] If two elements of $P$ have exactly $m$ elements
that they both cover, then there will be exactly $m$ elements
that cover them both.
\end{enumerate}
Favorite examples include the $1$-differential
poset {\it Young's lattice} $\Y$ of all number partitions ordered by inclusion of their
Ferrers diagrams, the $r$-differential Cartesian product order $Y^r$, and the
Fibonacci posets $Z(r)$ discussed in  below.  See Stanley \cite{S1, S2}
for more on differential posets, and Fomin \cite{Fomin1, Fomin2} 
for the more general notion of {\it dual graded graphs}.

There is an important rephrasing of the properties (D1, D2) above.
Let $P_n$ denote the $n^{th}$ rank of $P$, say with rank size $p_n:=|P_n|$.
For any ring $R$, consider the free $R$-module $R^{p_n}$ having basis indexed by
the elements of $P_n$, and define the {\it up} and {\it down} maps
$$
\begin{aligned}
U_n (=D_{n+1}^t): R^{p_n} & \rightarrow R^{p_{n+1}} \\
D_n (=U_{n-1}^t): R^{p_n} & \rightarrow R^{p_{n-1}} 
\end{aligned}
$$
in which a basis element is sent by $U_n$ (resp. $D_n$)
to the sum, with coefficient $1$, of all elements that cover (resp. are covered by) it.
Using the abbreviated notations
$$
\begin{aligned}
UD_n &:=U_{n-1}D_n \\
DU_n &:=D_{n+1}U_n
\end{aligned}
$$
the conditions (D1, D2) can be rephrased as saying that, for each $n \geq 1$,
$$
DU_n-UD_n=rI
$$
or equivalently, that
\begin{equation}
\label{r-differential-relation}
DU_n=UD_n+rI.
\end{equation}

This turns out to imply that, working over rings $R$ of characteristic zero,
$U_n, D_n$ are of maximal rank, via a simple calculation of the
eigenvalues and multiplicities for the nonnegative definite operators $DU_n, UD_n$;
see Proposition~\ref{eig} below.  The main conjecture of this paper,
Conjecture~\ref{conj:parametrized-Smith} below, gives further information
about the maps  $DU_n, UD_n$ (as well as $U_n, D_n$ themselves)
which is independent of the ring $R$.

For a ring $R$ and an $m \times n$ matrix $A$ in $R^{m \times n}$
say that $A$ 
\begin{enumerate}
\item[$\bullet$] {\it is in Smith form}
if it is diagonal in the sense that $a_{ij} = 0$ for
$i \neq j$, and its diagonal entries $s_i:=a_{ii}$
have $s_i$ dividing $s_{i+1}$ in $R$ for each $i$.
\item[$\bullet$]
{\it has a Smith form over $R$}
if it can be brought into Smith form by change-of-bases in the
free modules $R^m, R^n$, that is, if there exist matrices
$P$ in $GL_m(R)$ and $Q$ in $GL_n(R)$ for which $PAQ$ is in Smith form.
\item[$\bullet$]
{\it has a parametrized Smith form for $A+tI$}
if $A$ is square and the matrix $A+tI$ in $R[t]^{n \times n}$ has a Smith form over $R[t]$.
\end{enumerate}

We will mainly be interested in the case where $R=\ZZ$.  In this situation,
since $\ZZ$ is a principal ideal domain (PID), having a Smith form for a matrix $A$ over $R$ 
is guaranteed, but most square matrices $A$ over $R$ will {\it not} have a 
parametrized Smith form for $A+tI$, as $\ZZ[t]$ is not a PID.  
Having such a parametrized Smith form leads to very strong consequences, including a prediction
for the Smith forms over $\ZZ$ for all specializations $A+kI$;  see the Appendix
(Section~\ref{sec:Smith-significance}) below.

\begin{conjecture}
\label{conj:parametrized-Smith}
In an $r$-differential poset, for every $n \geq 0$,
$DU_n$ in $\ZZ^{p_n \times p_n}$ has a parametrized Smith form for $DU_n+tI$
over $\ZZ[t]$.
\end{conjecture}
\noindent

Note that due to equation \eqref{r-differential-relation},
it would be equivalent to conjecture that
$UD_n$ when $n \geq 1$ has a parametrized Smith form for $UD_n+tI$, 
as the ring automorphism of $\ZZ[t]$ that maps $t \mapsto t+r$ 
would convert one Smith form to the other.

The paper is structured as follows.

Section~\ref{sec:two-related-conjectures} discusses two conjectures (Conjectures~\ref{conj:numerical}
and \ref{conj:free-cokernel}) on $r$-differential 
posets that are closely related to Conjecture~\ref{conj:parametrized-Smith}.

Section~\ref{sec:computer-tests} discusses some computer experiments checking
consequences of Conjecture~\ref{conj:parametrized-Smith}.

Section~\ref{sec:constructions} discusses two standard constructions for
producing $r$-differential posets, and their behavior with respect to
the conjectures.

Section~\ref{sec:Fibonacci-section} proves all of the conjectures
for the Young-Fibonacci poset $\YF$ and its $r$-differential generalizations
$Z(r)$.

Section~\ref{sec:Young} discusses  Young's lattice $\Y$ and the
$r$-fold Cartesian products $\Y^r$.  Although we have not been able
to prove Conjecture~\ref{conj:parametrized-Smith} in these cases,
some of its predictions are verified here.

Section~\ref{sec:dual-graded-graphs}
discusses some remarks and questions regarding attempts to generalize parts of
the conjecture to the {\it dual graded graphs} of Fomin \cite{Fomin1}.

The appendix Section~\ref{sec:Smith-significance} 
discusses general facts about Smith forms over rings which
are not necessarily PID's, and the consequences of having a parametrized
Smith form for $A+tI$.

%%%%%%%%%%%%%%%%%%%%%%%%%%%%%%%%%%%%%%%%%%%%%%%%%%%%%%%%%
\section{Two related conjectures}
\label{sec:two-related-conjectures} 
%%%%%%%%%%%%%%%%%%%%%%%%%%%%%%%%%%%%%%%%%%%%%%%%%%%%%%%%%

In all of the conjectures in this section,
$P$ is an $r$-differential poset, with rank sizes $p_n$, 
up maps $U_n$ and down maps $D_{n+1}=U_n^t$ for $n\geq 0$.

\subsection{On the rank sizes}
\label{sec:numerical}

It turns out that in the presence of a certain numerical condition on the rank sizes $p_n$
(Conjecture~\ref{conj:numerical} below), the main Conjecture~\ref{conj:parametrized-Smith} 
implies and is closely related to a weaker conjecture (Conjecture~\ref{conj:free-cokernel} below).

We recall here an observation of Stanley concerning how the rank sizes $p_n$ in a differential poset 
determine the spectra of the operators $DU_n$ and $UD_n$.

\begin{proposition} \cite[\S 4]{S1}
\label{eig}  
Let $P$ be an $r$-differential poset.  Then
\[
\begin{aligned}
\det(DU_n+tI) &=\prod_{i=0}^n (t+r(i+1))^{\Delta p_{n-i}} \\
\det(UD_n+tI) &=\prod_{i=0}^n (t+ri)^{\Delta p_{n-i}}
\end{aligned}
\]
where $\Delta p_n:=p_n - p_{n-1}$.

In particular, as $DU$ has all positive integer eigenvalues, 
\begin{enumerate}
\item[$\bullet$]
$DU_n$ is nonsingular if $R$ has characteristic zero,
\item[$\bullet$]
the up maps $U_n$ are injective if $R$ has characteristic zero,
\item[$\bullet$]
the down maps $D_n$ are surjective if $R$ has characteristic zero, and
\item[$\bullet$]
the rank sizes weakly increase: $p_0 \leq p_1 \leq p_2 \leq \cdots$.
\end{enumerate}
\end{proposition}

Since one knows the eigenvalues of $UD_n$, and they lie in $\ZZ$,
Proposition~\ref{prop:smith-eigenvalue-relation} below shows that
whenever Conjecture~\ref{conj:parametrized-Smith} holds, the Smith normal
form entries $s_1(t),\ldots,s_{p_n}(t)$ for $DU_n+tI$ take the following
explicit form in terms of the rank sizes $p_n$:
\begin{equation}
\label{explicit-Smith-entries}
s_{p_n+1-j}(t) = \prod_{\substack{i=0 \\ \Delta p_{n-i} \geq j}}^n (t+r(i+1))
\end{equation}
for $j=1,2,\ldots,p_n$, where an empty product is taken to be $1$. 
In particular, this predicts for any chosen element
$k$ in the ring $R$ the structure of the kernel or cokernel of $DU_n+kI$ or $UD_n+kI$,
and if $R$ is a field, will predict the ranks of $DU_n+kI, UD_n+kI$
in terms of the characteristic of the field.

In light of Proposition~\ref{eig}, an affirmative answer to the following question of 
Stanley on the {\it strict} increase of the $p_n$, or positivity of the first differences $\Delta p_n$,
would determine the {\it support} of the spectra of the operators $DU_n$ and $UD_n$,
and hence determine the last Smith entry $s_{p_n}(t)$ of $DU_n+tI$.

\begin{question} (Stanley \cite[\S 6, end of Problem 6.]{S1}) \rm \ \\
\label{conj:Stanley}
In an $r$-differential poset $P$, do the rank sizes $p_n$ satisfy
$$
(r=)p_1 < p_2 < p_3 < \cdots?
$$
\end{question}

We make the following conjecture on the rank sizes.
\begin{conjecture}
\label{conj:numerical}
For every $n \geq 1$ and $1 \leq j \leq n$, one has
$$
\Delta p_n \geq \Delta p_{n-j-\delta_{r,1}} 
= 
\begin{cases}
\Delta p_{n-j-1} & \text{ if }P\text{ is }1\text{-differential,}\\
\Delta p_{n-j} & \text{ if }P\text{ is }r\text{-differential for }r \geq 2.
\end{cases} 
$$
In other words, for $r \geq 2$, one has
\begin{equation}
\label{second-differences-nonnegative}
\Delta p_1 \leq \Delta p_2 \leq \cdots
\end{equation}
and for $r=1$ one has
$$
\Delta p_1, \Delta p_2, \ldots, \Delta p_{n-3}, \Delta p_{n-2} \leq \Delta p_n.
$$
\end{conjecture}
\noindent
Conjecture~\ref{conj:numerical} has the following useful
rephrasing, which is equivalent via Proposition~\ref{eig}:
the zero eigenvalue of $UD_n$ achieves the largest multiplicity among
all of its {\it non-unit} eigenvalues.

We remark that for many known examples 
of $r$-differential posets $P$ 
(but not all;  see Section~\ref{sec:computer-tests} below), 
the stronger condition
\eqref{second-differences-nonnegative} also holds for $r=1$.
In other words, the {\it second differences} 
$$
\Delta^2 p_n : = \Delta p_n - \Delta p_{n-1} = p_n - 2p_{n-1} + p_{n-2}
$$
are nonnegative.  In this case, one can conveniently recompile
the Smith entries of $DU_n+tI$ over $\ZZ[t]$ from \eqref{explicit-Smith-entries} 
and express their multiplicities in terms of these second differences:
\begin{center}
\begin{tabular}{c|c}
entry & multiplicity\\
\hline
\hline
$
\begin{cases}
(n+1)!_{r,t} &\text{ if }r \geq 2\\
(n+1+t) \cdot (n-1)!_{1,t} &\text{ if }r =1
\end{cases}$
             & $1$\\
\hline
$(i+1)!_{r,t}$ & $\Delta^{2}p_{n-i}$  for $0 \leq i \leq n-1$\\
\hline
$1$ & $p_{n-1}$ 
\end{tabular}
\end{center}
where
\[
\ell !_{r,t}:=(r\cdot \ell+t)(r\cdot (\ell-1)+t)\cdots(r\cdot 1+t)
\]
and where by convention we set $p_k=0$ if $k$ is negative.

\subsection{On the up and down maps themselves}

\begin{conjecture}
\label{conj:free-cokernel}
For every $n$, the matrix $U_n$ in $\ZZ^{p_{n+1} \times p_n}$ 
has only ones and zeroes in its Smith form over $\ZZ$,
that is, $\coker(U_n: \ZZ^{p_n} \rightarrow \ZZ^{p_{n+1}})$
is free abelian.
\end{conjecture}

\noindent
Note that it would be equivalent to conjecture that
$D_{n+1}$ has free cokernel, or only ones and zeroes in
its Smith form of $\ZZ$, since $D_{n+1}=U_n^t$.

\begin{proposition}
\label{imp:prop}
Whenever Conjecture~\ref{conj:numerical} holds,
then Conjecture~\ref{conj:parametrized-Smith} implies Conjecture~\ref{conj:free-cokernel}.
\end{proposition}
\begin{proof}
One can be more precise about the relationship:
whenever Conjecture~\ref{conj:numerical} holds,
one has equivalence between the assertion of Conjecture~\ref{conj:free-cokernel} 
and the correctness of the Smith form over $\ZZ$ for $UD_n$ which is predicted by setting $t=0$
in the Smith form over $\ZZ[t]$ for $UD_n+tI$ that follows from
Conjecture~\ref{conj:parametrized-Smith} and Proposition~\ref{prop:smith-eigenvalue-relation}.

To see this, recall that Proposition~\ref{eig} allows one to
rephrase Conjecture~\ref{conj:numerical} as saying that the
the zero eigenvalue for $UD_n$ achieves the largest multiplicity among
all of its non-unit eigenvalues.
Thus whenever this holds, Proposition~\ref{prop:smith-eigenvalue-relation} says
that the Smith form over $\ZZ$ for $UD_n$ predicted by
Conjecture~\ref{conj:parametrized-Smith} will have only zeroes and ones as the
diagonal entries.  This is equivalent to 
$\coker(UD_n) = \ZZ^{p_n - p_{n-1}}$, which by Proposition~\ref{prop:free-cokernel-characterization}
is equivalent to $\coker(U)=\ZZ^{p_n-p_{n-1}}$, that is, to Conjecture~\ref{conj:free-cokernel}.
\end{proof}

\begin{remark}
Note that Conjecture~\ref{conj:free-cokernel} is equivalent 
(via Proposition~\ref{prop:det-divisors})
to saying that the $\gcd$ of all of the maximal minor subdeterminants in 
$U_n$ or $D_{n+1}$ is $1$.  In fact, many examples exhibit a strictly stronger
property:  $U_n$ has at least one maximal minor subdeterminant {\it equal} to $\pm 1$.
We suspect that this stronger property does {\it not} hold for all $r$-differential
posets, although we have not yet encountered any counterexamples
%
%  Alex checked this for all 643 different isomorphism classes of 
% 1-differential posets up through rank 8, and found no counterexamples.
%
\end{remark}

%%%%%%%%%%%%%%%%%%%%%%%%%%%%%%%%%%%%%%%%%%%%%%%%%%%%%%%%
\section{Computer tests}
\label{sec:computer-tests}
%%%%%%%%%%%%%%%%%%%%%%%%%%%%%%%%%%%%%%%%%%%%%%%%%%%%%%%%

We report here on a few numerical experiments.

The first author and P. Byrnes used a computer to
find all possible {\it partial $1$-differential posets} (in the sense described
in Section~\ref{sec:Wagner} below) 
\begin{enumerate}
\item[$\bullet$] of rank $8$, coming in $643$ poset isomorphism classes, and 
\item[$\bullet$] of rank $9$, coming in $44605$ poset isomorphism classes.
\end{enumerate}
Any of these partial $1$-differential posets can be extended to a full
$1$-differential poset by Wagner's construction from the next section, 
and therefore represent all the  possibilities for the first $8$ or $9$ ranks.

The data up through rank $9$ exhibits three isomorphism classes
where the first differences of the rank numbers are {\it not} weakly
increasing;  see the remark at the end of Section~\ref{sec:numerical} for
the significance.  All three of these isomorphism classes
classes have rank sizes
$$
(p_0,p_1,p_2,\ldots,p_9)=(1,1,2,3,5,7,11,15,26,36)
$$
whose first differences are
$$
(\Delta p_1,\Delta p_2,\ldots,\Delta p_9)=(0,1,1,2,2,4,4,11,10)
$$
so that $\Delta_8 = 11 > 10 =\Delta_9$.

Regarding Conjecture~\ref{conj:parametrized-Smith}, if it were correct,
then the Smith form over $\ZZ[t]$ would have
its diagonal Smith entries $s_i(t)$ given explicitly by \eqref{explicit-Smith-entries},
via Proposition~\ref{eig} and Proposition~\ref{prop:smith-eigenvalue-relation}.
As mentioned in Remark~\ref{fraction-field-and-PID-consequences} of the Appendix, 
this then predicts the Smith form over $\ZZ$ (which always exists) for $DU_n+kI$ 
when $k$ is any chosen integer.
For each of the $643$ non-isomorphic
partial $1$-differential posets of rank $8$, considering $DU_n+kI$
with $n=0,1,2,\ldots,7$, this predicted Smith form over $\ZZ$ 
was checked to be correct for all integers $k$ in the range $[-10,+10]$.
In addition, in Young's lattice $\Y$, the Smith form for $DU_n (=DU_n+0 \cdot I)$
was checked to obey the prediction up through $n=20$.

%%%%%%%%%%%%%%%%%%%%%%%%%%%%%%%%%%%%%%%%%%%%%%%%%%%%%%%%
\section{Constructions for $r$-differential posets}
\label{sec:constructions} 
%%%%%%%%%%%%%%%%%%%%%%%%%%%%%%%%%%%%%%%%%%%%%%%%%%%%%%%%%

We recall here two known ways to construct new differential posets
from old ones, and discuss how these constructions respect Conjectures~\ref{conj:parametrized-Smith}, 
\ref{conj:numerical}, and \ref{conj:free-cokernel}.

\subsection{Wagner's reflection-extension construction}
\label{sec:Wagner}
\emph{Wagner's construction} produces an $r$-differential posets from a
\emph{partial $r$-differential poset} of some finite rank (see~\cite[\S 6]{S1}).
Let $P$ be a finite graded poset of rank $n$ with a minimum element,
and assume that for $i=0,1,\ldots,n-1$ one has
\[
DU_i-UD_i=rI.
\]
One calls $P$ a \emph{partial $r$-differential poset} of rank $n$.  
Let $P^+$ be the poset of rank $n+1$ obtained from $P$ in the following way: for each $v\in P_{n-1}$, 
add a vertex $v^*$ of rank $n+1$ to $P$ that covers exactly those $x\in P_n$ covering $v$.  
Finally, above each $x\in P_n$ we adjoin $r$ new elements covering only $x$ in $P^+$.  
It is not hard to see that iterating this construction produces an $r$-differential poset.

An important example of this construction are the posets $Z(r)$ described explicitly
in Section~\ref{sec:Fibonacci-section} below, but which can also be obtained 
by applying Wagner's construction to the partial poset of rank zero having
only one element.
%
%\begin{center}
%\label{shrub-figure}
%{
%\psfrag{dts}{$\cdots$}
%\psfrag{1}{$1$}
%\psfrag{r}{$r$}
%\includegraphics{shrub.eps}
%}
%\end{center}  
As an example, see Figure~\ref{z2} for $Z(2)$.

 \begin{figure}[hbtp]
{
\includegraphics{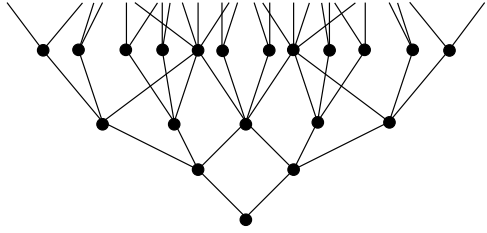}
}
\caption[]{Part of the $2$-differential poset $Z(2)$.}\label{z2}
\end{figure}

\begin{proposition}
\label{prop:wagner-respects-conjectures}  
Let $\hat{P}$ be a partial $r-$differential poset of rank $n$, and $P$ the differential
poset obtained from $\hat{P}$ by Wagner's construction.  
\begin{enumerate}
\item[(i)]
If the inequalities of Conjecture~\ref{conj:numerical} are satisfied by $\hat{P}$ up through
rank $n$, then $P$ satisfies Conjecture~\ref{conj:numerical}.
\item[(ii)]
If Conjecture~\ref{conj:free-cokernel} is satisfied by $\hat{P}$, in the
sense that $\coker(U_j)$ is free for all $0\leq j\leq n-1$,
then $P$ satisfies Conjecture~\ref{conj:free-cokernel}.
\end{enumerate}
\end{proposition}
\begin{proof}  
For (i), note that
\[
p_{n+1}=p_{n-1}+r p_{n}
\]
and hence
$$
\Delta p_{n+1} = (r-1)p_n + p_{n-1}.
$$
Therefore if $r \geq 2$ then
$$
\Delta p_{n+1} \geq p_n  \geq \Delta p_n
$$
giving Conjecture~\ref{conj:numerical} in this case.
If $r=1$ then it says
$$
\Delta_{p_{n+1}} = p_{n-1}
\geq p_{n-j} \geq \Delta p_{n-j} 
$$
for $j \geq 1$, which gives Conjecture~\ref{conj:numerical} in this case.

For (ii), note that by Wagner's construction, the $p_{n+1} \times p_n$ matrix
for $U_n$ will contain a $p_n \times p_n$ (maximal) minor equal to the identity
matrix, and hence have free cokernel.
\end{proof}

\begin{remark}
The proof of Theorem~\ref{thm:Fibonacci} below will show that, for $r \geq 2$,
Wagner's construction also respects Conjecture~\ref{conj:parametrized-Smith},
in the following sense:  if one has a partial $r$-differential 
poset with ranks $0,1,\ldots,n-1$, and one applies Wagner's construction twice
producing the two new ranks $P_n, P_{n+1}$,
then one has the following relation (see \eqref{Fibonacci-recursive-smith}
below):
$$
DU_n+tI
\quad \sim \quad
I_{n-1} \oplus (t+r)I_{(r-2)\cdot p_{n-1}+p_{n-2}} 
\oplus (t+r)\left(DU_{n-1}+(t+r)I\right)
$$
where $A \oplus B$ denotes block direct sum of matrices,
and $A \sim B$ for matrices $A,B$ in $\ZZ[t]^{m \times n}$ means that
$P(t) \cdot A \cdot Q(t)=B$ for some $P(t),Q(t)$ in $GL_m(\ZZ[t]), GL_n(\ZZ[t])$ respectively.
Thus $DU_n+tI$ has Smith form over $\ZZ[t]$ if and only if
$DU_{n-1}+tI$ does.

Unfortunately the part of the proof of Theorem~\ref{thm:Fibonacci} 
which treats the case $r=1$ does not have this nice property.
\end{remark}

\subsection{The Cartesian product construction}

If $P$ and $P'$ are posets, their \emph{Cartesian product} $P\times P'$ is
endowed with the componentwise partial order.

\begin{example}  The following is a small example for the Cartesian product of two posets.
\begin{center}
{
\psfrag{X}{$\times$}
\psfrag{=}{$\cong$}
\includegraphics{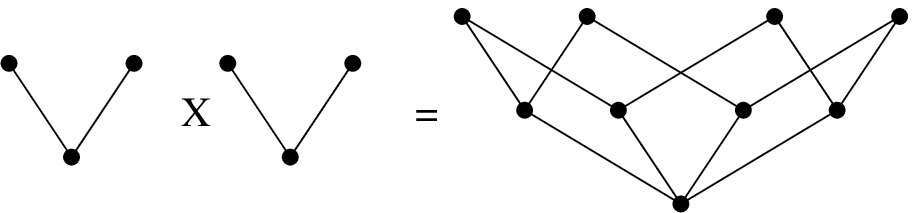}
}
\end{center}
\end{example}

Stanley observed that if $P$ and $P'$ are $r$- and $r'$-differential posets, then $P\times P'$ 
is an $(r+r')$-differential poset.

\begin{proposition}\label{prop:product-respects-numerical}  
If  $P, P'$ are any $r$-, $r'$-differential posets, then
Conjecture~\ref{conj:numerical} holds for the $(r+r')$-differential poset $P\times P'$.  
\end{proposition}

\begin{proof}  
As $r+r' \geq 2$, one needs to show that the rank sizes $r''_n:=|(P \times P')_n|$
satisfy $\Delta^2 r''_n \geq 0$.  If
$$
\begin{aligned}
P(t)&:=\sum_n r_n t^n\\
P'(t)&:=\sum_n r'_n t^n\\
\text{ then }(P\times P')(t)&:=\sum_n r''_n t^n=P(t)P'(t).
\end{aligned}
$$
Since both $P, P'$ have their rank sizes weakly increasing by Proposition~\ref{eig},
both $(1-t)P(t)$ and $(1-t)P'(t)$ lie in $\NN[[t]]$.
Hence their product
\[
(1-t)^2P(t)P'(t) = 
(1-2t+t^2) \sum_n r''_n t^n
\]
also lies in $\NN[[t]]$, which means that $\Delta^2 r''_n\geq 0$.
\end{proof}

\begin{proposition}\label{prop:product-respects-free-cokernel}  
Let $P$ and $P'$ be differential posets.  If Conjecture~\ref{conj:free-cokernel} holds
for either $P$ or $P'$ then it also holds for $P\times P'$.
\end{proposition}

\begin{proof}
We begin by noting that 
one can identify
$$
\ZZ^{(P \times P')_n} = \bigoplus_{i+j=n} \ZZ^{p_i} \otimes_\ZZ \ZZ^{p'_j}.
$$
Under this identification, the up map 
$U^{P \times P'}_{n-1}: \ZZ^{(P \times P')_{n-1}} \rightarrow \ZZ^{(P \times P')_n}$
is given by 
\begin{equation}
\label{product-up-map}
\bigoplus_{i+j=n} \left( I_{p_i} \otimes U'_j \right) \oplus
               \left( U_i \otimes I_{p'_j} \right)
\end{equation}
or in block matrix terms:
\[
U^{P \times P'}_{n-1}=
\left[
\begin{matrix}
  I_{p_0}\otimes U'_{n-1}    &                           &       &        &\\
  U_0\otimes I_{p'_{n-1}}    &I_{p_1}\otimes U'_{n-2} &      &        &\\
                              &U_1\otimes I_{p'_{n-2}}  &\ddots &        &\\
                              &                           &\ddots &\ddots  &\\
                              &                           &       &\ddots  & I_{p_{n-1}}\otimes U'_0\\
                              &                           &       &        & U_{n-1}\otimes I_{p'_0}
\end{matrix}
\right].
\]
Assume Conjecture~\ref{conj:free-cokernel} holds for $P$ (respectively, for $P'$).
Then one can perform row and column operations to arrive at a matrix
with the same zero blocks, but where the $(i,i)$ block of the form $I_{p_i} \otimes U'_{n-i-1}$
(respectively, the $(i+1,i)$ block of the form $U_{i} \otimes I_{p'_{n-i-1}}$ )
has been changed to 
\begin{equation}
\label{new-block}
I_{p_i} \otimes \left[ \begin{matrix} I_{p'_{n-i-1}} \\ 0 \end{matrix} \right].
\quad \left( \text{ respectively, }
\left[ \begin{matrix} I_{p_i} & 0 \end{matrix} \right] \otimes I_{p'_{n-i-1}}.
\right).
\end{equation}
One can then see that no matter how the factor below it (respectively, to the right of it)
has been altered during this process, one can then use the new block in \eqref{new-block} 
to eliminate it with row (respectively, column) operations. 
Thus the whole matrix has only zeroes and ones in its Smith form over $\ZZ$, and
hence has free cokernel, as desired.
\end{proof}

\begin{remark}  
The expression \eqref{product-up-map} for the up map in $P \times P'$ can
also be used to give the following block tridiagonal form for $DU_n$ in
the Cartesian product $P \times P'$:
\[
DU_n^{P\times P'}
 =
D_n \otimes U'_n
\quad \oplus \quad
\left( I_{p_n} \otimes DU'_n + DU_n \otimes I_{p'_n} \right)
\quad \oplus \quad 
U_n \otimes D'_n.
\]
However, we have not seen how to use this to show that
the Cartesian product construction respects 
Conjecture~\ref{conj:parametrized-Smith} in some way.
\end{remark}

%%%%%%%%%%%%%%%%%%%%%%%%%%%%%%%%%%%%%%%%%%%%%%%%%%%%%%%%%
\section{The Young-Fibonacci posets}
\label{sec:Fibonacci-section} 
%%%%%%%%%%%%%%%%%%%%%%%%%%%%%%%%%%%%%%%%%%%%%%%%%%%%%%%%%

As a set, the \emph{Fibonacci $r$-differential poset} 
$Z(r)$ consists of all finite words using alphabet 
$\{1_1,1_2,\ldots,1_r,2\}$.  For two words $w,w'\in Z(r)$, we define $w$ to cover $w'$ if either
\begin{enumerate}
\item $w'$ is obtained from $w$ by changing a $2$ to some $1_i$, as long as only $2$'s occur to its left, or
\item $w'$ is obtained from $w$ by deleting its first letter of the form $1_i$.
\end{enumerate}

\begin{figure}[hbtp]
{
\psfrag{0}{$\varnothing$}
\psfrag{1}{$1$}
\psfrag{2}{$2$}
\psfrag{12}{$12$}
\psfrag{21}{$21$}
\psfrag{111}{$111$}
\psfrag{112}{$112$}
\psfrag{121}{$121$}
\psfrag{211}{$211$}
\psfrag{22}{$22$}
\psfrag{1111}{$1111$}
\psfrag{11}{$11$}
\includegraphics{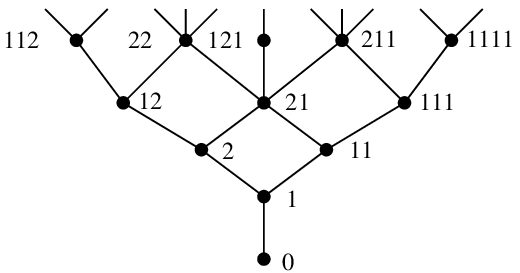}
}
\caption[]{The Young-Fibonacci lattice $\YF$.}\label{Z:lattice}
\end{figure}

It is known and not hard to show \cite[\S6]{S1} that
$Z(r)$ is obtained using Wagner's construction
from a single point
% poset in Figure~\ref{shrub-figure}, 
and hence an $r$-differential poset;  see
The reason for its name is that the rank sizes of $Z(1)=\YF$ (the \emph{Young-Fibonacci} lattice) 
are the Fibonacci numbers; see Figure~\ref{Z:lattice}.

\subsection{Proof of all conjectures for $\YF$ and $Z(r)$}

\begin{theorem}  
\label{thm:Fibonacci}
All of the Conjectures~\ref{conj:parametrized-Smith}, \ref{conj:numerical}, \ref{conj:free-cokernel} 
hold for the Young-Fibonacci $r$-differential posets $Z(r)$ with $r \geq 1$.
\end{theorem}

\begin{proof}
It follows from Proposition~\ref{prop:wagner-respects-conjectures}
that Conjectures~\ref{conj:numerical} and \ref{conj:free-cokernel} hold.   
For Conjecture~\ref{conj:parametrized-Smith}, one
begins with the following recursive descriptions of $U_n$ and $D_n$:
\begin{equation}
\label{recursive-up-down-descriptions}
U_n=
\begin{bmatrix}
D_n\\
I_{p_n}\\
\vdots\\
I_{p_n}
\end{bmatrix}\qquad\text{and}\qquad
D_{n+1}=
\begin{bmatrix}
U_{n-1} & I_{p_n} & \cdots & I_{p_n}
\end{bmatrix}.
\end{equation}
Note that each block matrix above contains exactly $r$ blocks of $I_{p_n}$.  
Using this, one can write 
\[D_{n+1}=\
\begin{bmatrix}
\begin{matrix} D_{n-1}\\I_{p_{n-1}}\\\vdots \\I_{p_{n-1}}\end{matrix} & 
\begin{matrix} I_{p_{n-2}} &      &       &\\
                       &I_{p_{n-1}}&       &\\
                       &      &\ddots &\\
                       &      &       &I_{p_{n-1}}
\end{matrix} & \cdots & 
\begin{matrix} I_{p_{n-2}} &      &       &\\
                       &I_{p_{n-1}}&       &\\
                       &      &\ddots &\\
                       &      &       &I_{p_{n-1}}
\end{matrix}
\end{bmatrix}
\]
%\qquad\text{and}\]
%\[U_n=
%\begin{bmatrix}
%\begin{matrix}
%U_{n-2}& I_{p_{n-1}} &\cdots & I_{p_{n-1}}
%\end{matrix}\\
%\begin{matrix}
%I_{p_{n-2}}&     &       &    \\
%      & I_{p_{n-1}} &        &    \\
%      &     &\ddots &    \\
%      &     &       & I_{p_{n-1}} 
%\end{matrix}\\
%\vdots\\
%\begin{matrix}
%I_{p_{n-2}}&     &        &    \\
%      & I_{p_{n-1}} &        &    \\
%      &     &\ddots &    \\
%      &     &       & I_{p_{n-1}} 
%\end{matrix} 
%\end{bmatrix}.\]
and multiplying it by its transpose $U_n$, one obtains the following
$(r+1) \times (r+1)$ block matrix for $n \geq 2$:
\begin{equation}
\label{Fibonacci-starting-point}
\begin{aligned}
&DU_n+(t-r)I_{p_n}=\\
&\begin{bmatrix}
DU_{n-2}+tI_{p_{n-2}} &D_{n-1} &\cdots & \cdots & \cdots & D_{n-1}\\
U_{n-2}          &(t+1)I_{p_{n-1}} & I_{p_{n-1}} & \cdots & \cdots & I_{p_{n-1}} \\
\vdots & I_{p_{n-1}} & \ddots & \ddots & &\vdots \\
\vdots & \vdots &\ddots & \ddots &\ddots &\vdots\\
\vdots & \vdots & &\ddots &\ddots & I_{p_{n-1}}\\
U_{n-2} & I_{p_{n-1}} & \cdots &\cdots &I_{p_{n-1}} & (t+1)I_{p_{n-1}}
\end{bmatrix}.
\end{aligned}
\end{equation}

We now use this description of $DU_n+(t-r)I_{p_n}$ 
together with elementary row and column operations to show
that it has a Smith form in $\ZZ[t]$.  Given two matrices $A, B$ in $\ZZ[t]^{m \times n}$,
say that $A \sim B$ if there exist $P(t), Q(t)$ in $GL_m(\ZZ[t]), GL_n(\ZZ[t])$ with $PAQ=B$.

\vskip.1in
\noindent
{\sf Case 1. $r \geq 2$}

\noindent
Because $r \geq 2$, in \eqref{Fibonacci-starting-point} there is a copy of $I_{p_{n-1}}$ 
in the second row and last column.  Use it to first clear out the blocks in
the second row, then to clear out the blocks in the last columns, and lastly,
negate the second column:
$$
\begin{aligned}
&DU_n+(t-r)I_{p_n} \\
%&=&
%\begin{bmatrix}
%DU_{n-2}+tI_{p_{n-2}} &D_{n-1} &\cdots & \cdots & \cdots & D_{n-1}\\
%U_{n-2}          &(t+1)I_{p_{n-1}} & I_{p_{n-1}} & \cdots & \cdots & I_{p_{n-1}} \\
%\vdots & I_{p_{n-1}} & \ddots & \ddots & &\vdots \\
%\vdots & \vdots &\ddots & \ddots &\ddots &\vdots\\
%\vdots & \vdots & &\ddots &\ddots & I_{p_{n-1}}\\
%U_{n-2} & I_{p_{n-1}} & \cdots &\cdots &I_{p_{n-1}} & (t+1)I_{p_{n-1}}
%\end{bmatrix}\\
&\sim
\begin{bmatrix}
tI_{p_{n-2}} &-tD_{n-1} &0      & \cdots & \cdots & 0 &D_{n-1}\\
0       &0         & 0     &  &  & \vdots&I_{p_{n-1}} \\
0       & -tI_{p_{n-1}} & tI_{p_{n-1}}     & \ddots & &\vdots &I_{p_{n-1}} \\
0       & -tI_{p_{n-1}} & 0           &   \ddots      & \ddots &\vdots & \vdots\\
\vdots  & \vdots   &\vdots & \ddots &tI_{p_{n-1}} &0 &I_{p_{n-1}}\\
0& -tI_{p_{n-1}} &0 & \cdots & 0& tI_{p_{n-1}}&I_{p_{n-1}}\\
%0 & -tI_{p_{n-1}} & 0&\cdots &0 & I_{p_{n-1}}\\
-tU_{n-2} & -t(t+2)I_{p_{n-1}} & -tI_{p_{n-1}} &\cdots &-tI_{p_{n-1}} &-tI_{p_{n-1}} &(t+1)I_{p_{n-1}}
\end{bmatrix}\\
&\sim
\begin{bmatrix}
tI_{p_{n-2}} &tD_{n-1} &0      & \cdots & \cdots & 0 &0\\
0       &0         & 0     &  &  & \vdots&I_{p_{n-1}} \\
0       & tI_{p_{n-1}} & tI_{p_{n-1}}     & \ddots & &\vdots &0 \\
0       & tI_{p_{n-1}} & 0           &   \ddots      & \ddots &\vdots & \vdots\\
\vdots  & \vdots   &\vdots & \ddots &tI_{p_{n-1}} &0 &0\\
0& tI_{p_{n-1}} &0 & \cdots & 0& tI_{p_{n-1}}&0\\
%0 & -tI_{p_{n-1}} & 0&\cdots &0 & I_{p_{n-1}}\\
-tU_{n-2} & t(t+2)I_{p_{n-1}} & -tI_{p_{n-1}} &\cdots &-tI_{p_{n-1}} &-tI_{p_{n-1}} &0
\end{bmatrix}
\end{aligned}
$$
After rearranging the rows and columns, this yields
$$
I_{p_{n-1}}\oplus
\begin{bmatrix}
       tD_{n-1} &0      & 0 & \cdots & 0 &tI_{p_{n-2}}\\
       tI_{p_{n-1}} & tI_{p_{n-1}}     & 0 & \cdots&0 &0 \\
        tI_{p_{n-1}} & 0           &   \ddots      & \ddots &\vdots & \vdots\\
        \vdots   &\vdots & \ddots &\ddots &0 &0\\
tI_{p_{n-1}} &0 & \cdots & 0& tI_{p_{n-1}}&0\\
 t(t+2)I_{p_{n-1}} & -tI_{p_{n-1}} &\cdots &-tI_{p_{n-1}} &-tI_{p_{n-1}} &-tU_{n-2}
\end{bmatrix}
$$
Subtracting columns $2,3,\ldots,r-1$ each from the first column yields
\begin{eqnarray*}
I_{p_{n-1}}&\oplus&
\begin{bmatrix}
tD_{n-1} &0      & \cdots & \cdots & 0 &tI_{p_{n-2}}\\
 0 & tI_{p_{n-1}}     & 0 & \cdots&0 &0 \\
  \vdots & 0           &   \ddots      & \ddots &\vdots & \vdots\\
  \vdots   &\vdots & \ddots &\ddots &0 &\vdots\\
0 &0 & \cdots & 0& tI_{p_{n-1}}&0\\
t(r+t)I_{p_{n-1}} & 0 &\cdots &\cdots &0 &-tU_{n-2}
\end{bmatrix}\\
&\sim&
I_{p_{n-1}}\oplus tI_{(r-2)\cdot p_{n-1}}\oplus
\begin{bmatrix}
                    tD_{n-1}     & tI_{p_{n-2}} \\
                     t(r+t)I_{p_{n-1}}& -tU_{n-2} 
\end{bmatrix}\\
&\sim&
I_{p_{n-1}} \oplus tI_{(r-2)\cdot p_{n-1}} 
\oplus tI_{p_{n-2}} \oplus
t\left((r+t)I_{p_{n-1}}+UD_{n-1}\right) 
\\
&=&
I_{p_{n-1}}\oplus tI_{(r-2)\cdot p_{n-1}+p_{n-2}}\oplus  t(DU_{n-1}+tI_{p_{n-1}}),
\end{eqnarray*}
for $n\geq 2$.

That is, when $r \geq 2$, one has for $n\geq 2$ that
\begin{equation}
\label{Fibonacci-recursive-smith}
DU_n+tI\sim I_{p_{n-1}} \oplus (t+r)I_{(r-2)\cdot p_{n-1}+p_{n-2}} \oplus
(t+r)\left(DU_{n-1}+(t+r)I\right).
\end{equation}
This shows that $DU_n+tI$ has a Smith form over $\ZZ[t]$ by induction on $n$.

\vskip.1in
\noindent
{\sf Case 2. $r=1$}

\noindent
For $r=1$, one proceeds slightly differently.  One can check
that $DU_n+tI$ has a Smith form over $\ZZ[t]$ in the base cases of $n=0,1,2,3$ directly.
Then for $n \geq 4$, starting with \eqref{Fibonacci-starting-point} for $r=1$, one has
$$
\begin{aligned}
&DU_n+(t-1)I_{p_n}\\
&=\left[
\begin{matrix}
DU_{n-2} + tI_{p_{n-2}} & D_{n-1} \\
U_{n-2}  & (t+1)I_{p_{n-1}}
\end{matrix}
\right] \\
 &=
\left[
\begin{matrix}
DU_{n-2} + tI_{p_{n-2}} & U_{n-3} & I_{p_{n-2}} \\
D_{n-2}  & (t+1)I_{p_{n-3}}  & 0 \\
I_{p_{n-2}}  & 0      &(t+1)I_{p_{n-2}}
\end{matrix}
\right] 
\end{aligned}
$$
in which the second step used the recursive descriptions
\eqref{recursive-up-down-descriptions} of $U_n, D_n$.
Use the $I_{p_{n-2}}$ in the upper-right block to clear out the first row, 
and then the first column:
$$
\begin{aligned}
&DU_n+(t-1)I_{p_n}\\
 &\sim
\left[
\begin{matrix}
0  & 0 & I_{p_{n-2}} \\
D_{n-2}  & (t+1)I_{p_{n-3}}  & 0 \\
-(t+1)DU_{n-2}-(t^2+t-1)I_{p_{n-2}}  & -(t+1)U_{n-3}   &(t+1)I_{p_{n-2}}
\end{matrix}
\right] \\
 &\sim
\left[
\begin{matrix}
0  & 0 & I_{p_{n-2}} \\
D_{n-2}  & (t+1)I_{p_{n-3}}  & 0 \\
-(t+1)DU_{n-2}-(t^2+t-1)I_{p_{n-2}}  & -(t+1)U_{n-3}   & 0
\end{matrix}
\right]\\
 &\sim
I_{p_{n-2}} \oplus A
\end{aligned}
$$
where 
$$
A:=
\left[
\begin{matrix}
D_{n-2}  & (t+1)I_{p_{n-3}} \\
-(t+1)DU_{n-2}-(t^2+t-1)I_{p_{n-2}}  & -(t+1)U_{n-3}
\end{matrix}
\right]
$$
Next perform the column operation which subtracts from the first
column the result of right-multiplying the second column
by $D_{n-2}$.  Using the fact that $UD_{n-2}=DU_{n-2}-I_{p_{n-2}}$, this yields
a matrix
$$
A_n(t):=
\left[
\begin{matrix}
-tD_{n-2}  & (t+1)I_{p_{n-3}} \\
-t(t+2)I_{p_{n-2}}  & -(t+1)U_{n-3}
\end{matrix}
\right]
$$

It only remains to show that $A_n(t)$ has a Smith form over $\ZZ[t]$, which we
do by induction on $n$.  
Start by using the recursive descriptions \eqref{recursive-up-down-descriptions}
of $D_n, U_n$ to rewrite $A_n(t)$ as
$$
\begin{aligned}
A_n(t)
&:=
\left[
\begin{matrix}
-tU_{n-4}  & -tI_{p_{n-3}} &(t+1)I_{p_{n-3}} \\
-t(t+2)I_{p_{n-4}} & 0  & -(t+1)D_{n-3} \\
0    & -t(t+2)I_{p_{n-4}} & -(t+1)I_{p_{n-3}} 
\end{matrix}
\right] \\
&\sim
\left[
\begin{matrix}
-tU_{n-4}  & -tI_{p_{n-3}} & I_{p_{n-3}} \\
-t(t+2)I_{p_{n-4}} & 0  & -(t+1)D_{n-3} \\
0    & -t(t+2)I_{p_{n-4}} & -(t^2+3t+1)I_{p_{n-3}} 
\end{matrix}
\right]
\end{aligned}
$$
where the latter equivalence comes from adding the second column
to the third.  Now one can use the $I_{p_{n-3}}$ in the upper
right block to clear out the first row, and then clear out
the last column.  The end result shows that
$$
A_n(t)\sim
I_{n_3} \oplus
(-t)\left[
\begin{matrix}
(t+2)I_{p_{n-4}}+(t+1)DU_{n-4}  & (t+1)D_{n-3} \\
(t^2+3t+1)U_{n-4}  & (t^2+4t+3)U_{n-3}
\end{matrix}
\right].
$$
If one now subtracts from the first column the
result of right-multiplying the second column
by $U_{n-4}$, one concludes that
$$
\begin{aligned}
A_n(t)
&\sim
I_{n_3} \oplus
(-t)\left[
\begin{matrix}
(t+2)I_{p_{n-4}}  & (t+1)D_{n-3} \\
-(t+2)U_{n-4}  & (t+1)(t+3)U_{n-3}
\end{matrix}
\right]\\
&\sim
I_{n_3} \oplus
(-t) A_{n-1}(t+1).
\end{aligned}
$$
Since $A_{n-1}(t)$ (and hence also $A_{n-1}(t+1)$) 
has a Smith form over $\ZZ[t]$ by induction, so does $A_n(t)$.
\end{proof}

%%%%%%%%%%%%%%%%%%%%%%%%%%%%%%%%%%%%%%%%%%%%%%%%%%%%%%%%%
\section{Young's lattice and its Cartesian products}
\label{sec:Young} 
%%%%%%%%%%%%%%%%%%%%%%%%%%%%%%%%%%%%%%%%%%%%%%%%%%%%%%%%%

\begin{figure}[hbtp]
{
\psfrag{0}{$\varnothing$}
\includegraphics{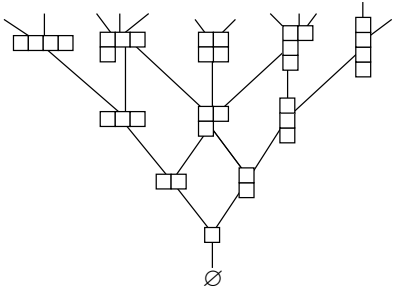}
}
\caption[]{Young's lattice.}\label{Y:lattice}
\end{figure}

The prototypical example of a $1$-differential poset is \emph{Young's lattice}, denoted by $\Y$.  
This is the set of all number partitions, ordered by inclusion of Young diagrams; see Figure~\ref{Y:lattice}. 
Its $r$-fold Cartesian product $\Y^r$ is then an $r$-differential poset.
These give the primary example of a set-up which we define next.

\begin{definition}
\label{tower-of-groups-defn}
For $r$ a positive integer, define an {\it $r$-differential tower of groups} to
be an infinite tower of finite groups
$1=G_0\leq G_1\leq G_2\leq\cdots$ such that
the branching rules for restricting irreducibles from $G_n$ to $G_{n-1}$ are
{\it multiplicity-free}, and satisfy 
\begin{equation}\label{eq:ind:res}
\Res^{G_{n+1}}_{G_n}\Ind_{G_n}^{G_{n+1}}-\Ind_{G_{n-1}}^{G_n}\Res^{G_n}_{G_{n-1}}=r\cdot{\rm id}.
\end{equation}
Here the maps in question are considered as operators 
on the $\ZZ$-module of virtual characters $\Class_\ZZ[G_n]$, that is,
the $\CC$-valued class functions on $G_n$ that are $\ZZ$-linear combinations of
irreducible characters for $G_n$.
\end{definition}

Each $r$-differential tower of groups gives rise to an $r$-differential poset $P$, whose
$n^{th}$ rank $P_n$ is labelled by $\Irr(G_n)$, the set of irreducible representations
of $G_n$.  One includes a covering relation between a $G_{n-1}$-irreducible $\chi$ and a
$G_n$-irreducible $\psi$ exactly when 
$\Res^{G_n}_{G_{n-1}} \psi$ contains $\chi$ (with multiplicity one)
as an irreducible constituent.  Identifying $\ZZ^{P_n}$ with $\Class_\ZZ[G_n]$,
equation \eqref{eq:ind:res} becomes $DU_n - UD_n = rI$.

\begin{example}
Let $G$ be a finite abelian group of order $r$, and let
$G_n$ be the wreath product $G \wr \symm_n = \symm_n[G]$ of $G$ and
the symmetric group $\symm_n$.  An irreducible 
representation of $G_n$ can naturally be indexed by an element of $\Y^r$
at rank $n$, that is, by an $r$-tuple $(\lambda^{(1)},\ldots,\lambda^{(r)})$ of partitions
whose total weight $\sum_i |\lambda^{(i)}|$ equals $n$.  Furthermore,
the branching rule from $G_n$ to $G_{n-1}$ is multiplicity free:  
one obtains a sum of irreducibles indexed by 
all $r$-tuples obtained by removing one box from any of the $\lambda^{(i)}$;
see Macdonald \cite[Chapter 1, Appendix B]{M}, Okada \cite[Corollary to Theorem 4.1]{Okada},
and Stembridge \cite[Lemma 7.1]{Stembridge}.
%
%Irreducible representations of $G_n$ have a known indexing (see e.g. 
%Macdonald \cite[Chapter 1, Appendix B]{M}, or Stembridge \cite{Stembridge}) by
%$r$-tuples $(\lambda^{(1)},\ldots,\lambda^{(r)})$ of whose total
%weight $\sum_i |\lambda^{(i)}|$ equals $n$.  Furthermore, the branching rule
%from $G_n$ to $G_{n-1}$ is multiplicity free:  one obtains a sum of 
%all $r$-tuples obtained by removing one box from any of the $\lambda^{(i)}$;
%see \cite[Lemma 7.1]{Stembridge}.  
%
Thus the irreducibles and branching
in this tower of groups are described by the $r$-differential poset $\Y^r$,
and hence this is an $r$-differential tower of groups.
\end{example}

\begin{remark}
A recent result of Bergeron, Lam, and Li \cite{BergeronLamLi} implies that
every $r$-differential tower of groups that satisfies some mild extra
hypotheses will have $|G_n|=r^n n!$.  This means that these
structures are more rigid than one might expect.
\end{remark}

For the purposes of some of the proofs in this section, we
briefly review in the case $r=1$, so that $G_n=\symm_n$, some
facts about bases of $\Class_\ZZ[\symm_n]$ and their labelling
by symmetric functions as in \cite{M} and \cite[Chapter 7]{S4};
see these references for more details, as well as \cite[Remark on p. 926]{S1}

The irreducible (complex) characters of $\symm_n$ are indexed by partitions
$\lambda$ of $n$, and are associated via Frobenius' characteristic map
to the Schur functions $s_\lambda$.  The above discussion
says that the branching map $D_n$ sends the basis element indexed by $s_\lambda$
to the sum of all $s_\mu$ where $\mu$ is obtained from $\lambda$ by removing
one cell from the Ferrers diagram.  On the other hand, there is another
$\ZZ$-basis for $\Class_\ZZ[\symm_n]$ indexed by partitions $\lambda$
of $n$, given by the characters of the coset representations for Young subgroups
$\symm_\lambda:=\symm_{\lambda_1} \times \cdots \times \symm_{\lambda_\ell}$
inside $\symm_n$.  The characters of these coset representations
correspond under the Frobenius characteristic map with the symmetric functions 
$
h_\lambda := h_{\lambda_1} \cdots h_{\lambda_\ell}
$
where $h_r$ is the complete homogeneous symmetric function of degree $r$.
In this basis, since the induction map $U_n$ multiplies by $s_1=h_1$, it
sends $h_\lambda$ to $h_{(\lambda,1)}$.  Also, the branching rule 
$D_n$ sends $h_n$ to $h_{n-1}$ and then extends as a
derivation on products.  From this one can derive the following special
case of the Mackey formula: if $\lambda=(\lambda,\ldots,\lambda_\ell)\vdash n$, 
then

\begin{equation}
\label{Mackey}
DU(h_\lambda)=D (h_1 h_\lambda)
=h_\lambda+h_1\sum_{i=1}^\ell h_{(\lambda_1,\ldots,\lambda_i-1,\ldots,\lambda_\ell)}
=h_\lambda+\sum_{i=1}^\ell h_{(\lambda_1,\ldots,\lambda_i-1,\ldots,\lambda_\ell,1)}.
\end{equation}

\subsection{Proof of Conjectures~\ref{conj:numerical} and \ref{conj:free-cokernel} for $\Y^r$}

\begin{proposition}
Conjecture~\ref{conj:numerical} holds for $\Y^r$
\end{proposition}
\begin{proof}
By Proposition~\ref{prop:product-respects-numerical}, 
it is enough to check it for $r=1$, that is, for $\Y$.
We will show the stronger assertion that
\begin{equation}
\label{Young-has-stronger-assertion}
\Delta p_1 \leq \Delta p_2 \leq \cdots.
\end{equation}
Recall that $p_n$ counts the number of partitions of $n$.
One can inject the partitions of $n-1$ into the partitions of $n$
by adding a part of size $1$.  Hence $\Delta p_n = p_n - p_{n-1}$ 
counts the number of partitions of $n$ that contain no parts of size $1$.  

One can then inject such partitions with no parts of size $1$ for
$n-1$ into those for $n$, by adding $1$ to the largest part.
This shows $\Delta p_{n-1} \leq \Delta p_n$.
\end{proof}

\begin{proposition}\label{prop:U:Y}  
Conjecture~\ref{conj:free-cokernel} holds for $\Y^r$
\end{proposition}
\begin{proof}
By Proposition~\ref{prop:product-respects-free-cokernel}, 
it is enough to check it for $\Y$.  But when one uses the
$\ZZ$-basis for $\Class_\ZZ(\symm_n)$ indexed by $\{h_\lambda\}$ as discussed above,
then the fact that $U_n(h_\lambda)=h_{(\lambda,1)}$ shows that $U_n$ has
free cokernel.
\end{proof}

\subsection{Ones in the Smith form of $DU+kI$ in $\Y$}

Conjecture~\ref{conj:parametrized-Smith} predicts for each integer $k$ 
the Smith normal form over $\ZZ$ for $DU+kI$; see Proposition~\ref{prop:smith-eigenvalue-relation}.
In particular, it predicts the number of {\it ones} that should appear in these Smith normal forms.
We check here that those predictions are correct for Young's lattice $\Y$.

\begin{proposition}
\label{prop:ones-right}
In Young's lattice $\Y$, for every integer $k$, 
the number of ones in the Smith normal form over $\ZZ$ of $DU_n+kI$ is
\begin{equation}
\label{predicted-ones}
\begin{cases}
p_{n-1} & \text{ for }k \neq 0,-2 \\
p_n-p_{n-1}+p_{n-2} & \text{ for }k =0,-2,
\end{cases}
\end{equation}
which agrees with the prediction from Conjecture~\ref{conj:parametrized-Smith}.
\end{proposition}
\begin{proof}
Let $\ones(A)$ for an integer matrix $A$ denote the number of
ones in its Smith normal form over $\ZZ$. 

First we check that the numbers given in \eqref{predicted-ones}
are the same as the ones predicted by the conjecture.
Proposition~\ref{prop:smith-eigenvalue-relation} predicts that 
$\ones(DU_n+kI)$ should be the difference between $p_n$
and the highest multiplicity attained by a non-unit (that is, not $\pm 1$) eigenvalue
of $DU_n+kI$.  Because $\Y$ satisfies the inequalities \eqref{Young-has-stronger-assertion},
Proposition~\ref{eig} shows that the highest multiplicity
of a non-unit eigenvalue will either be the multiplicity 
\begin{enumerate}
\item[$\bullet$]
$p_n-p_{n-1}$ attained by $k+1$ whenever $k+1 \neq \pm 1$, or
\item[$\bullet$]
$p_{n-1}-p_{n-2}$ attained by $k+2$.
\end{enumerate}  
Thus the predictions agree with \eqref{predicted-ones}.

To show that the numbers in \eqref{predicted-ones} actually match $\ones(DU_n+kI)$,
we express $DU_n+kI$ with respect to two different (ordered) bases of $\Class_\ZZ(\symm_n)$.

First use the $\{s_\lambda\}$ or irreducible character basis, with the partitions
ordered lexicographically.  One can see that it contains a $p_{n-1} \times p_{n-1}$ lower
unitriangular submatrix whose rows (resp. columns) are indexed by the
partitions $(\mu_1,\ldots,\mu_\ell,1)$ (resp. $(\mu_1+1,\mu_2,\mu_3,\ldots,\mu_\ell)$)
where $\mu$ runs through all partitions of $n-1$.  For example,
when $n=5$, the matrix looks as follows, with the rows and columns of the
aforementioned submatrix indicated by asterisks:
$$
\bordermatrix{
      &  \scriptstyle 5   & \scriptstyle 41^*   &\scriptstyle 32 &\scriptstyle 311^* &\scriptstyle 221^* 
           &\scriptstyle 2111^* &\scriptstyle 11111^* \cr
\scriptstyle 5^*     &  k+2   & 1    & 0  & 0     & 0     & 0     & 0\cr
\scriptstyle 41^*    &  1   & k+3    & 1  & 1     & 0     & 0     & 0\cr
\scriptstyle 32^*    &  0   & 1    & k+3  & 1     & 1     & 0     & 0\cr
\scriptstyle 311^* &  0   & 1    & 1  & k+3     & 1     & 1     & 0\cr 
\scriptstyle 221 &  0   & 0    & 1  & 1     & k+3     & 1     & 0\cr
\scriptstyle 2111^* &  0   & 0    & 0  & 1     & 1     & k+3     & 1\cr
\scriptstyle 11111   &  0   & 0    & 0  & 0     & 0     & 1     & k+2 
}
$$
Note that this shows, via Proposition~\ref{prop:det-divisors}, that for
any integer $k$ one has the inequality $\ones(DU_n+kI) \geq p_{n-1}$.

Alternatively, one can express $DU_n+kI$ 
in the $\{h_\lambda\}$-basis, ordered so that the partitions
$\lambda$ having no part of size $1$ come first.  Then
equation~\ref{Mackey} shows that it will have this block upper-triangular form:
\begin{equation}
\label{Y:DU+kI}
\left[\begin{matrix} 
(k+1)I_{p_n-p_{n-1}} & 0 \\
& \\
 * & [DU_{n-1}]+(k+1)I_{p_{n-1}}
\end{matrix}\right].
\end{equation}
For $n=5$, this looks as follows:
\[
\bordermatrix{
      &  \scriptstyle 5   &\scriptstyle 32   &\scriptstyle 41 &\scriptstyle 2^2 1 &\scriptstyle 3 1^2 &\scriptstyle 2 1^3 &\scriptstyle 1^5 \cr
\scriptstyle 5    &   k+1   &      &    &       &       &       &  \cr
\scriptstyle 32    &  0   & k+1    &    &       &       &       &  \cr
\scriptstyle 41    &  1   & 0    & k+2  &       &       &       &  \cr
\scriptstyle 2^2 1 &  0   & 1    & 0  & k+2     &       &       &  \cr 
\scriptstyle 3 1^2 &  0   & 1    & 1  & 0     & k+3     &       &  \cr
\scriptstyle 2 1^3 &  0   & 0    & 0  & 2     & 1     & k+4     &  \cr
\scriptstyle 1^5   &  0   & 0    & 0  & 0     & 0     & 1     & k+6 
},
\]
Since the first $p_n-p_{n-1}$ rows of the matrix for $DU_n+kI$ in \eqref{Y:DU+kI} have every
entry divisible by $k+1$, any minor subdeterminant of size $p_{n-1}+1$ will be divisible by
$k+1$.  Consequently, by Proposition~\ref{prop:det-divisors},
the Smith entry $s_{p_{n-1}+1} \neq 1$ as long as $k+1$ is not a unit
of $\ZZ$, that is, as long as $k \neq 0,-2$.  This shows
$\ones(DU_n+kI) \leq p_{n-1}$ in this case, and hence
$\ones(DU_n+kI) = p_{n-1}$ in light of the above lower bound.

Now if $k=0$ or $-2$ then \eqref{Y:DU+kI} shows
$$
\ones(DU+kI)=p_n-p_{n-1}+\ones(DU_{n-1}+(k+1)I).
$$
Therefore it only remains to check that for $k=0$ and for $k= -2$ 
one has 
$
\ones(DU_{n-1}+(k+1)I)=p_{n-2}:
$
\begin{enumerate}
\item[$\bullet$] For $k=0$ this was just proven above, replacing
$n$ by $n-1$ and taking $k=1$.
\item[$\bullet$]
For $k=-2$ one has 
$$
\ones(DU_{n-1}-I)=\ones(UD_{n-1})=p_{n-2}
$$
where the second equality holds because Propositions~\ref{prop:U:Y} and 
\ref{prop:free-cokernel-characterization} tell us that $UD_{n-1}$ has only
zeroes and ones in its Smith form.
\end{enumerate}
\end{proof}

\begin{remark}
It seems likely that one can extend the statement of Proposition~\ref{prop:ones-right}
and its proof from $\Y$ to $\Y^r$, using the analogues of the $\{h_\lambda\}$ 
basis for $\Class_\ZZ(G_n)$ where $G_n=\ZZ/r\ZZ \wr \symm_n$, as discussed in
\cite[Chapter 1, Appendix B, \S8]{M}.
However, we have not seriously attempted to carry this out.  

One slight annoyance is that the 
general prediction implied by Conjecture~\ref{conj:parametrized-Smith} 
for the number of ones when $r \geq 2$ will have 
a further special case in its statement:  when $r=2$ and $k=-3$,
the two eigenvalues $r+k=-1$ and $2r+k=+1$ achieving the highest multiplicities
{\it both} turn out to be units in $\ZZ^\times$.  
Thus Conjecture\ref{conj:parametrized-Smith} predicts, via Proposition~\ref{prop:smith-eigenvalue-relation},
that
$$
\ones(DU_n+kI)
\begin{cases}
p_{n-1} & \text{ for }r+k \neq \pm 1 \\
p_n-p_{n-1}+p_{n-2} & \text{ for }r+k = \pm 1 \text{ and }(r,k) \neq (2,-3)\\
p_n-p_{n-2}+p_{n-3} & \text{ for }(r,k)=(2,-3).
\end{cases}
$$
\end{remark}

\begin{remark}
Is there hope to extend the previous results to $r$-differential towers of group algebras,
or to $r$-differential self-dual Hopf algebras, as discussed in \cite{BergeronLamLi}?

In fact, Conjecture~\ref{conj:free-cokernel} is essentially
asking for something like an $\{h_\lambda\}$ basis in these contexts.
$U_n$ having free cokernel is equivalent to saying that there exists
a choice of bases in the source and target so that that
map $U_n: \ZZ^{P_n} \rightarrow \ZZ^{P_{n+1}}$ sends basis elements
into basis elements, as in $U_n(h_{\lambda})=h_{(\lambda,1)}$.
\end{remark}

\subsection{Last Smith entry of $DU+kI$ in $\Y^r$ and towers of group algebras}
\label{tower-section}

We wish to show that for $\Y^r$, and more generally, for any $r$-differential poset
coming from a tower of finite groups, the last Smith invariants $s_{p_n}$ for either of
the operators $DU_n+kI$ or $UD_n+kI$ match the prediction provided by Conjecture~\ref{conj:parametrized-Smith}
via Proposition~\ref{prop:smith-eigenvalue-relation}.

The idea is to use Corollary~\ref{cor:last-smith}.  We start with some facts
about the spectrum and inverse of $UD$ when $D:=\Res^G_H$ and $U=\Ind^G_H$ are restriction and induction of 
class functions for {\it any} pair of nested finite groups $H \subseteq G$.  Part (i)
of the next proposition is a generalization of Stanley's \cite[Proposition 4.7]{S1}
that was first observed by Okada \cite[Appendix, Theorem A]{Okada}.
%
%Part (i) was also conjectured to Vic by J. Shareshian in personal communication 
%about discrete Laplacians from 2001, but this was without knowing that Okada had
%already observed it 11 years earlier!
%

\begin{proposition}
\label{prop:ind-res-eigenvalues}
Let $G$ be a finite group, and $H \subseteq G$ a subgroup.  
\begin{enumerate}
\item[(i)]
The operator $UD$ on $\Class_\ZZ(G)$ has a complete system of orthogonal eigenvectors
given by
$$
v^{(g)}:=\sum_{\chi \in \Irr(G)} \chi(g) \cdot \chi
$$
as $g$ ranges through a set of representatives for the $G$-conjugacy classes.
The associated eigenvalue equation is
\begin{equation}
\label{eigenvalue-equation}
DU v^{(g)} = U(\triv_H)(g) \cdot v^{(g)}.
\end{equation}
\item[(ii)] 
Assume the integer $k$ makes the operator $A:=UD+kI$ nonsingular, that is,
the class function $\phi$ on $G$ defined by
$$
\phi(g):=U(\triv_H)(g)+k
$$
is nonvanishing.  Then the matrix for the inverse operator $A^{-1}$ has the
following $(\chi,\psi)$-entry for any $G$-irreducibles $\chi,\psi$:
$$
\begin{aligned}
A^{-1}_{\chi,\psi} 
  &= \left\langle \chi, \frac{\psi}{\phi} \right\rangle_G \\
\left( \right.
  &= \left\langle \chi \cdot \overline{\psi}, \frac{1}{\phi} \right\rangle_G 
  = \left\langle \frac{\chi}{\phi}, \psi \right\rangle_G 
  = A^{-1}_{\psi,\chi} 
\left. \right)
\end{aligned}
$$
where $\langle -, - \rangle_G$ denotes the usual inner product for
class functions on $G$.
\item[(iii)] Let $s_{|\Irr(G)|}$ be the last Smith normal form entry over $\ZZ$
for $UD+kI$. Then $s_{|\Irr(G)|}=0$ when $UD+kI$ is singular,
and when $UD+kI$ is nonsingular, $s_{|\Irr(G)|}$, it is the smallest positive integer
$s$ for which the class function $s \cdot \frac{1}{\phi}$ is
a virtual character (that is, an element of $\Class_\ZZ(G)$).
\end{enumerate}
\end{proposition}

\begin{proof}  
For (i), see the very simple proof in Okada \cite[Appendix, Theorem A]{Okada}.

For (ii), note that $A:=UD+kI$ is diagonalizable, with
the same complete system of eigenvectors $v^{(g)}$ having eigenvalues
$\phi(g)=U \triv_H(g) + k$.  Since {\it any} class function $\alpha$ on $G$ 
satisfies
$$
\sum_\chi \langle \chi, \alpha \rangle_G \cdot \chi(g) = \alpha(g),
$$
one can apply this to the class function $\alpha=\frac{\psi}{\phi}$
where $\psi$ is some $G$-irreducible character, and conclude that
$$
\sum_\chi \left\langle \chi, \frac{\psi}{\phi} \right\rangle_G \chi(g) = \frac{\psi(g)}{\phi(g)}.
$$
In other words, letting $B$ be the square matrix with rows and columns indexed by
$G$-irreducibles, and having entries
$$
B_{\chi,\psi}:= \left\langle \chi, \frac{\psi}{\phi} \right\rangle_G,
$$
one concludes that, for each $g$ in $G$,
$$
\sum_\chi B_{\chi,\psi} \cdot (v^{(g)})_\chi = \phi(g)^{-1} \cdot (v^{(g)})_\psi.
$$
In other words, $B v^{(g)} = \phi(g)^{-1} v^{(g)}$.  Since this is true
for every $g$ in $G$, one concludes that $B=A^{-1}$.

For (iii), note that Corollary~\ref{cor:last-smith} and part (ii) above
imply that $s_{|\Irr(G)|}$ is the smallest positive integer that makes
$$
s \left\langle \chi \cdot \overline{\psi}, \frac{1}{\phi} \right\rangle_G
=
\left\langle \chi \cdot \overline{\psi}, \frac{s}{\phi} \right\rangle_G
$$ 
an integer for all $G$-irreducibles $\chi, \psi$.  By taking $\psi=\triv$, one can see that
it is necessary for $s$ to make $\langle \chi , \frac{s}{\phi} \rangle_G$ an integer
for all $G$-irreducibles, that is, $\frac{s}{\phi}$ is a virtual character.  However,
this is also a sufficient condition on $s$, 
since $\chi \cdot \overline{\psi}$ is the character of a 
genuine (and hence also virtual) $G$-representation.
\end{proof}

We wish to apply the previous proposition with $H=G_{n-1}$ and $G=G_n$
in an $r$-differential tower of groups $1=G_0 < G_1 < \cdots$
as in Definition~\ref{tower-of-groups-defn}.  In this case, there is a nice
recurrence for the class function $\phi=\frac{1}{U(\triv_{n-1})+k}$ on $G_n$ that appears
in the proposition, where here $\triv_n$ is the 
trivial character of $G_n$, and $U, D$ denote induction $\Ind_{G_{n-1}}^{G_n}$
and restriction $\Res^{G_n}_{G_{n-1}}$ of class functions.  In particular,
transitivity of induction/restriction means that $U^j=\Ind_{G_{n-j}}^{G_n}$
and $D^j=\Res_{G_{n-j}}^{G_n}$ for all $j$ and $n$.

\begin{proposition}\label{eq:sum}
In an $r$-differential tower of groups, one has for every $n$ and every $k$ in $\CC$
the following equality among class functions in $\Class(G_n)$:
\[
\frac{1}{U(\triv_{n-1})+r+k}=\sum_{j=0}^n (-1)^j\frac{U^j(\triv_{n-j})}{(j+1)!_{r,k}}.
\]
where we recall that $j!_{r,k}:=(r+k) (2r+k) (3r+k) \cdots (jr+k)$.
\end{proposition}

\begin{proof}  
We begin with two simple facts.  First note that 
\eqref{eq:ind:res} implies
\begin{equation}
\label{first-simple-fact}
\begin{aligned}
D \left( \frac{1}{U(\triv_n)+r+k } \right)
&=  \frac{1}{DU(\triv_n)+r+k } \\
&=  \frac{1}{UD(\triv_n)+2r+k } \\
& = \frac{1}{U(\triv_{n-1})+2r+k}.
\end{aligned}
\end{equation}

Next, note that if $G$ is a finite group, and $H$ a subgroup, 
then for any class functions $\psi, \phi$ on $H, G$, respectively,
one can use Frobenius reciprocity to show
\begin{equation}
\label{Frob-recip-fact}
\Ind_H^G(\psi\cdot\Res^G_H\phi)=(\Ind_H^G\psi)\cdot\phi.
\end{equation}

We now show the asserted equality by induction on $n$.  In the base case $n=0$, one can
see that both sides are the constant class function $\frac{1}{r+k}$.  

In the inductive step, assume that the assertion holds for $n$, and manipulate
the sum on the right side of the proposition for $n+1$ as follows:
$$
\begin{aligned}
\sum_{j=0}^{n+1} (-1)^j \frac{U^j(\triv_{n+1-j})}{(j+1)!_{r,k}}
&=\frac{1}{r+k}-\sum_{j=1}^{n+1}(-1)^j\frac{U^j(\triv_{n+1-j})}{(j+1)!_{r,k}}\\
&\overset{(\ell:=j-1)}{=}
   \frac{1}{r+k}\left( 1 -U \sum_{\ell=0}^{n}(-1)^\ell \frac{U^\ell(\triv_{n-\ell})}{(\ell+1)!_{r,k+r}} \right)\\
&\overset{\text{(induction)}}{=}\frac{1}{r+k}\left( 1 -U \left( \frac{1}{U(\triv_{n-1})+2r+k}\right) \right)\\
&\overset{\eqref{first-simple-fact}}{=}
\frac{1}{r+k}\left( 1 -U \left( \triv_n \cdot D \left( \frac{1}{U(\triv_{n})+r+k}\right) \right) \right)\\
&\overset{\eqref{Frob-recip-fact}}{=}\frac{1}{r+k}\left( 1 - \frac{U(\triv_n)}{U(\triv_{n})+r+k} \right) \\
&=\frac{1}{r+k}\left( \frac{r+k}{U(\triv_{n})+r+k}\right) \\
&=\frac{1}{U(\triv_{n})+r+k}
\end{aligned}
$$
finishing the induction.
\end{proof}

We need one further general fact about $r$-differential posets that does not
involve groups.

\begin{proposition}
\label{prop:join-irreducible-chains}
Let $P$ be an $r$-differential poset $P$, and $j_{m-1} \lessdot j_m$ any
covering pair in $P$ with the property that 
$j_{m-1}, j_m$ both cover at most one element of $P$.

Then for any integer $n \geq m$ one can extend this to a saturated chain 
$$
j_{m-1} \lessdot j_m \lessdot j_{m+1} \lessdot \cdots \lessdot j_{n-1} \lessdot j_n
$$
in which each $j_\ell$ covers at most one element of $P$.
\end{proposition}
\begin{proof}
Induct on $n$, with base case $n=m$ trivial.
Since $j_n$ covers exactly one element, it is covered by exactly $r+1$ distinct
elements, say $a_1,a_2,\ldots,a_{r+1}$.  We wish to show that some $a_i$ covers
only $j_n$, and hence plays the role of $j_{n+1}$ in extending our chain.

So assume not, that is, there exists elements $a_1^-,\ldots,a_{r+1}^-$ all
distinct from $j_n$ with the property that $a_i$ covers $a_i^-$ for each $i$.
For each $i$, since $a_i^-$ and $j_n$ are covered by at least one element in
common, namely $a_i$, they must both cover some element in common, which must
be $j_{n-1}$.  Thus $j_n,a_1^-,\ldots,a_{r+1}^-$ all cover $j_{n-1}$, and our
inductive hypothesis says that $j_{n-1}$ covers at most one element, so it
is covered by at most $r+1$ elements.  This forces $a_i^-=a_j^-$ for some $i \neq j$.
But then since $j_n, a_i^- (=a_j^-)$ cover only one element $j_{n-1}$, they must
be covered by a unique element, forcing $a_i=a_j$, a contradiction.
\end{proof}

\begin{theorem}
\label{theorem:last-smith-entry}
Given an $r$-differential tower of groups, its $r$-differential poset has
the last Smith entry for $A:=DU_n+kI$ equal to $0$ when $A$ is
singular, and given by the formula
$$
s_{p_n} =
\begin{cases}
(n+1)!_{r,k}  & \text{ if }r \geq 2 \\
(n-1)!_{1,k}\cdot (n+1+k)& \text{ if }r =1
\end{cases}
$$
when $A$ is nonsingular.
\end{theorem}
\begin{proof}
By Proposition~\ref{prop:ind-res-eigenvalues}(iii),
it suffices to show that when $DU_n+kI$ is nonsingular,
the integers shown on the right side of the theorem are the smallest
positive integers $s$ for which the class function
$\frac{s}{\phi}$, where $\phi(g)=\Ind_{G_{n-1}}^{G_n}(\triv_{n-1}) + k$
is a virtual character of $G_n$.  

By Proposition~\ref{eq:sum}, this is same as
saying that
\[
s\cdot\sum_{j=0}^n(-1)^j\frac{\Ind_{G_{n-j}}^{G_n}\triv_{n-j}}{(j+1)!_{r,k}}
\]
is a virtual $G_n$-character. Because $j!_{r,k}$ divides $(j+1)!_{r,k}$, it is clear that
$(n+1)!_{r,k}$ is one such positive integer $s$, for any $r$.  Thus $s_{p_n}$ will
divide $(n+1)!_{r,k}$.  Also, since when
$r=1$ one has $G_0=G_1=1$, and since
$$
\frac{1}{n!_{1,k}} -
\frac{1}{(n+1)!_{1,k}}
=\frac{1}{(n-1)!_{1,k} (n+1+k)},
$$
one finds that $(n-1)!_{1,k}(n+1)_{1,k}$ is one such positive integer $s$ for $r=1$.
Thus $s_{p_n}$ will divide $(n-1)!_{1,k}(n+1+k)$ when $r=1$.

We wish to show that these divisibilities are actually equalities.  First
note every irreducible $G_n$-character appears in the irreducible decomposition
of $\Ind_{G_0}^{G_n} \triv_0$ since $G_0=1$ is the trivial subgroup of $G_n$.
Likewise for $\Ind_{G_1}^{G_n} \triv_1$ when $r=1$ since $G_1=G_0=1$ when $r=1$.
Hence, it suffices to show that 
there is at least one irreducible $G_n$-character $\chi$
which does not appear in the irreducible decomposition of 
$\Ind_{G_\ell}^{G_n} \triv_\ell$ for $\ell \geq 1$ if $r \geq 2$, and 
for $\ell \geq 2$ if $r=1$.  

This is provided by Proposition~\ref{prop:join-irreducible-chains}:
find a $G_n$-irreducible $j_n:=\chi$ lying atop a saturated chain 
$j_0 \lessdot j_1 \lessdot \cdots \lessdot j_n$,
each of whose elements covers a unique element of $P$, and which starts 
\begin{enumerate}
\item[$\bullet$]
for $r \geq 2$, in ranks $0,1$ with 
$j_0=\triv_0 \lessdot j_1=\psi$ for some nontrivial $G_1$-irreducible $\psi$,
\item[$\bullet$]
for $r =1$, in ranks $0,1,2$ with 
$j_0=\triv_0 \lessdot j_1=\triv_1 \lessdot j_1=\psi$
for the unique nontrivial $G_2$-irreducible $\psi$.
\end{enumerate}
In both cases, one can see that this character $\chi$ 
will have the desired property of not appearing in the appropriate
inductions $\Ind_{G_\ell}^{G_n} \triv_\ell$.
\end{proof}

\begin{corollary}
\label{cor:Stanley-and-last-smith-in-towers}
Given an $r$-differential tower of groups, 
\begin{enumerate}
\item[(i)]
the associated $r$-differential poset 
has affirmative answer to Stanley's Question~\ref{conj:Stanley}, that is,
its rank numbers $p_n$ satisfy $\Delta p_n > 0$ for $n \geq 2$, and
\item[(ii)]
for all integers $k$ the last Smith invariant $s_{p_n}$ of
$DU_n+kI$ obeys the prediction of Conjecture~\ref{conj:parametrized-Smith}:
it is zero when $DU_n+kI$ is singular, and equal to 
the product of the distinct eigenvalues of $DU_n+kI$ when it is nonsingular.
\end{enumerate}
In particular, these assertions apply to the $r$-differential Cartesian products
$\Y^r$ of Young's lattice.
\end{corollary}

\begin{proof}
For a fixed $n \geq 2$, to show that $\Delta p_n > 0$, 
apply Theorem~\ref{theorem:last-smith-entry} with the integer $k$
chosen so that $r+k$ is a prime that divides {\it none}
of 
$$
2r+k, \, 3r+k, \, \ldots, \, (n+1)r+k
$$
(e.g. pick $p$ to be a
prime larger than $(n+1)r$ and take $k=p-r$).
Theorem~\ref{theorem:last-smith-entry} tells us that the prime $r+k$ divides the last 
Smith form entry $s_{p_n}$ for $DU_n+kI$ over $\ZZ$,
and hence it also divides 
$$
\det(DU_n+kI) = 
(r+k)^{\Delta p_n} (2r+k)^{\Delta p_{n-1}} \cdots (nr+k)^{\Delta p_1} ((n+1)r+k)^{p_0}
$$
according to Proposition~\ref{eig}.
Since $r+k$ is a prime that can only divide the first factor on the right, it must be
that $\Delta p_n > 0$.

The last assertion of the corollary is now an immediate consequence of Theorem~\ref{theorem:last-smith-entry},
Proposition~\ref{eig}, and Proposition~\ref{prop:smith-eigenvalue-relation}
\end{proof}

\begin{remark}
\label{full-Okada-remark}

When $G$ is a {\it nonabelian} finite group, the up and down maps for the towers of
wreath products $G_n = \symm_n[G]=G \wr \symm_n$
were also studied by Okada \cite{Okada}.  In this situation,
one can see that the branchings are not multiplicity-free,
already from $G_1=G$ down to the trivial group $G_0$.
Hence one does not have an $r$-differential tower of
groups as we have defined it, and our proof of
Corollary~\ref{cor:Stanley-and-last-smith-in-towers} 
does not quite apply.  However, its conclusions for the up and
down maps are still valid, as we now explain.

In this situation, $\Irr(G_n)$ can be naturally indexed by
the elements at rank $n$ in $\Y^m$ where $m=|\Irr(G)|$ is the number
of complex irreducible $G$-characters, or the number of 
conjugacy classes in $G$.  Okada shows \cite[Corollary to Proposition 5.1]{Okada}
that one still has $DU_n-UD_n = rI$ in general.

Thus by placing the branching multiplicities on the edges of $\Y^m$,
one obtains the structure of a {\it self-dual graded graph} in the
sense of Fomin \cite{Fomin1, Fomin2}.  Our results such as
Proposition~\ref{eig}, \ref{prop:ind-res-eigenvalues} still apply,
with the same proofs.  However, in the proof of Theorem~\ref{theorem:last-smith-entry}
one needs a substitute for Proposition~\ref{prop:join-irreducible-chains},
which we currently don't know how to generalize to all self-dual graded graphs.

Nevertheless, for the wreath products $G_n$, one can easily produce
such a substitute as follows.  Take any nontrivial $G$-irreducible 
$\psi \neq \triv_G$ acting on a vector space $V$. Let
$\chi$ be the representation of $G_n$ acting
on the $n$-fold tensor product $V\otimes \cdots \otimes V$, in which
the the subgroup $\symm_n$ of $G_n$ acts by permuting tensor positions
and the Cartesian product subgroup $G^n$ acts as by $\psi$ within
each component.  This representation $\chi$ turns out to be a $G_n$-irreducible,
and its restriction to $G_1(=G)$ is a multiple of $\psi$.
Hence by Frobenius reciprocity, $\chi$ 
is not contained in the irreducible support of 
$\Ind_{G_1}^{G_n} \triv_{G_{1}}$, as desired.
\end{remark}

%%%%%%%%%%%%%%%%%%%%%%%%%%%%%%%%%%%%%%%%%%%%%%%%%%%%%%%%%
\section{Remarks on dual graded graphs}
\label{sec:dual-graded-graphs}
%%%%%%%%%%%%%%%%%%%%%%%%%%%%%%%%%%%%%%%%%%%%%%%%%%%%%%%%%

Our conjectures deal only with differential posets, not dual graded graphs,
because it has not been clear to us what the general setting should be.  

In particular, Remark~\ref{full-Okada-remark} perhaps suggests that
one should consider as an intermediate level of generality
the {\it self-dual graded graphs}.  We have not tested this extensively.

When considering dual graded graphs which are not self-dual, 
one finds that the naive generalizations of
Conjectures~\ref{conj:parametrized-Smith}, \ref{conj:numerical} and
\ref{conj:free-cokernel} fail; we present specific counterexamples
below.  However, we also present some examples
in which Conjecture~\ref{conj:free-cokernel} holds for at least
one of the two graphs in a pair of dual graded graphs.

Conjecture~\ref{conj:free-cokernel} fails already for the up and down maps in
the left graph out of the pair of dual graded graphs pictured in Figure~\ref{counter:U}.
\begin{figure}[hbtp]
{
\includegraphics{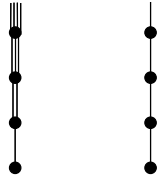}
}
\caption[]{Counterexample for Conjecture~\ref{conj:free-cokernel}}\label{counter:U}
\end{figure}

Conjecture~\ref{conj:numerical} fails for the
dual graded graphs $\SY$ of {\it shifted shapes} considered
by Fomin \cite{Fomin1, Fomin2}:  they have rank size $p_n$ equal to the number
of partitions of $n$ into distinct parts, for which $\Delta p_4=0 < 1=\Delta p_3$.

The dual graphs of shifted shapes also 
provide a counterexample to Conjecture~\ref{conj:parametrized-Smith},
when one considers $DU_7$.
\begin{figure}[hbtp]
{
\includegraphics{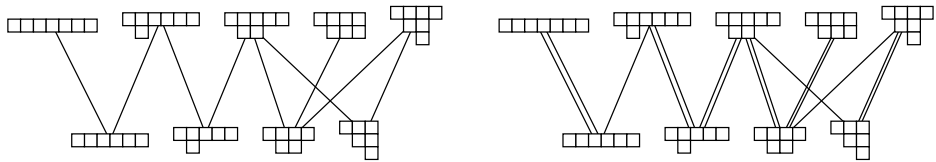}
}
\caption[]{Ranks $6$ and $7$ of $\SY$.}\label{counter:DU}
\end{figure}

One finds that
\[
DU_7=\begin{bmatrix} 4&2&2&0&0\\
1&3&2&0&0\\
1&2&5&2&0\\
0&0&2&4&2\\
0&0&0&1&3
\end{bmatrix}\sim \begin{bmatrix}1& & & & \\
 &1& & & \\
 & &1& & \\
 & & &2& \\
 & & & &120
\end{bmatrix},
\]
while its eigenvalues $1,2,3,5,8$ would have predicted the diagonal
Smith entries as $1,1,1,1,240.$

On the other hand, we list in the next proposition a number of known graded graphs, each
a part of a pair of dual graded graphs, having $\coker\ U_n$ free for all $n$; see
Nzeutchap \cite{Nz,Nz2} for definitions and background on these objects.

\begin{proposition}  
In the following graded graphs, $\coker\ U_n$ is free for all $n$:
\begin{enumerate}[(1)]
\item  the lifted binary tree,
\item  the bracket tree,
\item  Binword,
\item  the lattice of binary trees, and
\item  the lattice of shifted shapes from the pair $\SY$ with all edges labelled one.
\end{enumerate}
\end{proposition}

\begin{proof}  
For (1) and (2) this is clear because the graded graphs in question are trees.  It therefore remains to check
this for (3),(4),(5).  For each of these posets $P$, one can exhibit
\begin{enumerate}
\item[$\bullet$]
a natural way to linearly order each rank $P_n$, and 
\item[$\bullet$]
a ($P$-)order-preserving injection $P_n \overset{\phi}{\hookrightarrow} P_{n+1}$, so that
\item[$\bullet$]
$\phi(x)$ is always the earliest element in the linear order on $P_{n+1}$
which lies above $x$ in $P$.
\end{enumerate}
From this it follows that $U_n$ has free cokernel, because the sets $P_n$ and $\phi(P_n)$
will index the columns and rows of a $p_n \times p_n$ unitriangular submatrix.

For (3), let $w=w_1w_2\cdots w_n$ be a binary word of rank $n$ in Binword,
and linearly order them by the usual integer order on the numbers they represent in binary.
Then let $\phi(w)$ be obtained from $w$ by inserting a $0$ into the second position: 
\[\phi(w)=w_1 0 w_2\cdots w_n.\]  

For (4), linearly order each rank in the lattice of binary trees by saying
say $T\geq T'$ if the leftmost vertex in which they differ is in $T$.  For example, one has
\begin{center}
{
\psfrag{l}{$\geq$}
\includegraphics{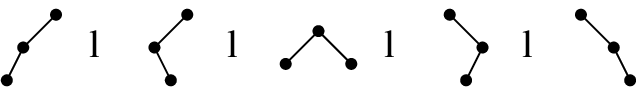}
}
\end{center}
Given a tree $T$, let $\phi(T)$ be obtained by adding a left leaf in the leftmost available position.

For the lattice of shifted shapes, the linear order on each rank
is the natural lexicographic ordering.  For example, 
\begin{center}
{
\psfrag{g}{$\geq$}
\includegraphics{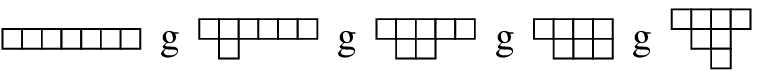}
}
\end{center}
Given a shifted shape $\lambda$ of rank $n$, 
let $\phi(\lambda)$ be obtained from it by removing the bottom, rightmost box.  
\end{proof}

\begin{question}
How should the conjectures in this paper generalize to dual graded graphs?  
\end{question}

We have noticed that in a large number of examples, the
torsion part of the group $\coker(U_n)$ is killed by
some power of the product of all edge-multiplicities
for the edges between ranks $n$ and $n+1$.
%
%\begin{equation}\label{conj:U}
%\left(
%  \prod_{\substack{ (x,y)\in E_1 \\ \rho(x)=n } }  m_1(x,y) 
%\right)^k
%\end{equation}
%for some exponent $k$.  
However, this also does not hold in general:
the following partial dual graded graphs
\begin{center}
\includegraphics{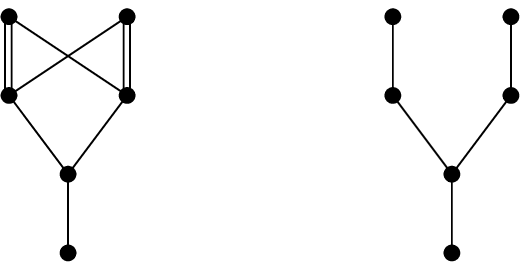}
\end{center}
have $\coker\ U_2\cong \ZZ/3\ZZ$, which is not killed by
any power of $2$ (the only nontrivial edge multiplicity between ranks $2$ and $3$).
One can extend this partial pair to a genuine pair of dual graded graphs by a
fairly obvious dual graded graph analogue of Wagner's construction.

%%%%%%%%%%%%%%%%%%%%%%%%%%%%%%%%%%%%%%%%%%%%%%%%%%%%%%%%%
\section{Appendix: On Smith and parametrized Smith forms}
\label{sec:Smith-significance} 
%%%%%%%%%%%%%%%%%%%%%%%%%%%%%%%%%%%%%%%%%%%%%%%%%%%%%%%%%

  This appendix discusses Smith forms over rings $R$ which are unique factorization
domains (UFD's), but not necessarily principal ideal domains (PID's).  See also
Kuperberg \cite{Kuperberg} for other useful facts about this, and 
interesting examples where such Smith forms exist.  
Also discussed here is the significance of a square matrix having
a {\it parametrized Smith form}, as defined in the Introduction.

\subsection{Smith forms}

Let $R$ be a commutative ring.  Recall that a matrix $A$ in $R^{m \times n}$
{\it is in Smith form} if it is diagonal in the sense that its off-diagonal entries
vanish, and its diagonal entries $s_i=a_{ii}$ have $s_{i+1}$ divisible by $s_i$ 
in $R$ for each $i=1,2,\ldots,\min\{m,n\}-1$.  Recall also that $A$ {\it has a Smith form over $R$}
if there exists change-of-bases $P, Q$ in $GL_m(R), GL_n(R)$, respectively, such that
$PAQ$ is in Smith form.

It is well-known that when $R$ is a PID, every matrix $A$ in $R^{m \times n}$ has
a Smith form, but this is not guaranteed for an arbitrary ring $R$.
The following proposition characterizing the diagonal entries in the Smith form
is well-known in the
case where $R$ is a PID, but holds more generally when $R$ is a UFD.

\begin{proposition}
\label{prop:det-divisors}
Let $R$ be a UFD, and assume that $A$ in $R^{m \times n}$ has a Smith form over $R$, with
diagonal entries $s_1,s_2,\ldots$.  Then for each $i$, the partial product
$s_1 s_2 \cdots s_i$ can be interpreted as the greatest common divisor $d_i(A)$
of all determinants of $i \times i$ minors of $A$.

Consequently, the entries $s_i$ are determined uniquely up to
units in $R^\times$ by the formula
\begin{equation}
\label{eqn:det-divisors}
s_i=\frac{d_i(A)}{d_{i-1}(A)},
\end{equation}
where $d_0(A)$ is taken to be $1$.
\end{proposition}
\begin{proof}
Since greatest common divisors in a UFD are well-defined up to units in $R^\times$,
the second assertion about unique determination follows once
one has shown the formula \eqref{eqn:det-divisors}

To show formula \eqref{eqn:det-divisors}, first note that it is trivial when $A$ itself is in
Smith form, due to the divisibility relations among the diagonal
entries $s_i$.

Secondly, note that $d_1(PAQ) = d_1(AQ) = d_1(A)$ since
invertible row or column operations over $R$ cannot affect the set of
common divisors of the entries of a matrix.  
Since we assumed that one can find $P,Q$ with
$PAQ$ in Smith form, the assertion holds for $i=1$.

For $i \geq 2$, recall that $A$ represents an 
$R$-module homomorphism between free $R$-modules
$$
A: R^n \rightarrow R^m.
$$
Apply the $i^{th}$ exterior power functor to obtain a matrix
$\wedge^i A$ in $R^{\binom{m}{i}} \times R^{\binom{n}{i}}$,
representing the $R$-module homomorphism between the free $R$-modules
$$
\wedge^i A: \wedge^i R^n \rightarrow \wedge^i R^m,
$$
with respect to the standard basis of decomposable wedges 
$e_I:=\wedge_{j \in I} e_i$ indexed by $i$-element subsets $I$ of 
$\{1,2,\ldots,n\}$ or $\{1,2,\ldots,m\}$ for the source and target.
Since the $(I,J)$ entry of $\wedge^i A$ is the determinant of the
$i \times i$ minor $A_{I,J}$, one concludes that $d_i(A)=d_1(\wedge^i A)$.
As functoriality implies that
$$
\wedge_i(PAQ) = \wedge^i P \cdot \wedge^i A \cdot \wedge^i Q,
$$
and also implies the invertibility of $\wedge^i P, \wedge^i Q$,
one concludes that
$$
d_i(PAQ) = d_1(\wedge^i(PAQ)) = d_1(\wedge^i P \cdot \wedge^i A \cdot \wedge^i Q) 
= d_1(\wedge^i A)=d_i(A).
$$
Since it was assumed that one can find $P,Q$ with
$PAQ$ in Smith form, the assertion holds for all $i$.
\end{proof}

The following corollary characterizing the last Smith form entry
is used in Section~\ref{tower-section}.

\begin{corollary}\label{cor:last-smith}
Let $R$ be a UFD, with field of fractions $K$.  Assume that a 
square matrix $A$ in $R^{n \times n}$ has a Smith form over $R$.
Then the last diagonal entry $s_n$ in its Smith form is
\begin{enumerate}
\item[$\bullet$] zero if and only if $A$ is singular, and
\item[$\bullet$] the smallest denominator for $A^{-1}$ in $K^{n \times n}$,
in the sense that any $s$ in $R$ for which $sA^{-1}$ lies in $R^{n \times n}$
will have $s_n$ dividing $s$.
\end{enumerate}
\end{corollary}

\begin{proof}
For the first assertion, note that since $A$ is square and has a Smith form $PAQ$
with diagonal entries $s_1,s_2,\ldots,s_n$, one has 
$$
s_1 \cdots s_n = \det PAQ = \det P \cdot \det A \cdot \det Q.
$$
Note that $\det P,\det Q$ lie in $R^\times$ since $P, Q$ invertible.
Hence $A$ is singular if and only if $\det A=0$ if and only if
$s_1 \cdots s_n=0$ if and only if some $s_i=0$ if and only if $s_n=0$.

For the second assertion, recall that
\[
(A^{-1})_{ij}=\frac{(-1)^{i+j}}{\det A} \det(A^{ij})
\]
where $A^{ij}$ is the $(n-1) \times (n-1)$ minor of $A$ obtained
by removing row $i$ and column $j$.
Thus one has
$$
\begin{aligned}
sA^{-1}\in R^{n\times n}
&\Leftrightarrow  s\frac{\det A^{ij}}{\det A} \in R \text{ for all }i,j\\
&\Leftrightarrow  \det A \mid s \det A^{ij} \text{ for all }i,j\\
&\Leftrightarrow \det A\mid \gcd_{i,j}(s \det A^{ij}) =s \gcd_{i,j}(\det A^{ij})\\
&\Leftrightarrow  s_1 s_2 \cdots s_{n-1} s_n = \det A \mid s d_{n-1}(A) = s \cdot s_1 s_2 \cdots s_{n-1} \\
&\Leftrightarrow  s_n \mid s,
\end{aligned}
$$
where Proposition~\ref{prop:det-divisors}
was used in the second-to-last equivalence.
\end{proof}

Assume $R$ is an integral domain with fraction field $K$, and
let $A$ in $R^{m \times n}$ have rank $r$ when considered as a matrix
in $K^{m \times n}$.  If $A$ also has a Smith form over $R$, then in its
diagonal entries, the first $r$ of them  $s_1,s_2,\ldots,s_r$ will be nonzero,
and the rest zero.   Note that the simplest possible Smith form
in this situation, where $s_1,\ldots,s_n \in \{0,1\}$, occurs if and only if the cokernel
$$
\coker A:= R^m / \im ( A : R^n \rightarrow R^m )
$$
is free, and isomorphic to $R^{m-r}$.  The following
equivalent characterization when rank $r=n$ is useful in the
context of Conjecture~\ref{conj:free-cokernel}.

\begin{proposition}
\label{prop:free-cokernel-characterization}
Let $R$ be an integral domain.
Assume $U$ in $R^{m \times n}$ has rank $n$ over
the fraction field $K$ of $R$, so in particular, $n \leq m$.
Then $\coker(U) \cong R^{m-n}$ if and only if $\coker(U U^t) \cong R^{m-n}$.
\end{proposition}

\begin{proof}
For the forward implication, if $\coker(U) \cong R^{m-n}$ then the
short exact sequence
$$
0 \rightarrow \im U \rightarrow R^m \rightarrow \coker(U) \rightarrow 0
$$
will split, since $\coker(U)$ is free, and hence projective.  Therefore
$$
R^m = \im(U) \oplus \coker(U)
$$
which means that there exist change-of-bases 
$P, Q$ in $GL_m(R), GL_n(R)$ so that 
$
PUQ=
\left[
\begin{matrix}
I_n \\
0
\end{matrix}
\right].
$
This gives 
$$
PU=
\left[
\begin{matrix}
I_n \\
0
\end{matrix}
\right] Q^{-1}
\quad = \quad
\left[
\begin{matrix}
Q^{-1} \\
0
\end{matrix}
\right] 
\quad = \quad
\left[
\begin{matrix}
Q^{-1} & 0\\
0      & I_{m-n}
\end{matrix}
\right] 
\left[
\begin{matrix}
I_n \\
0
\end{matrix}
\right].
$$
Therefore
$
\hat{P} U =
\left[
\begin{matrix}
I_n \\
0
\end{matrix}
\right]  
$
where
$
\hat{P}:=
\left[
\begin{matrix}
Q       & 0\\
0      & I_{m-n}
\end{matrix}
\right] P
$
is unimodular.
This implies 
$$
\hat P U U^t \hat{P}^t =
\left[
\begin{matrix}
I_n \\
0
\end{matrix}
\right]
\left[
\begin{matrix}
I_n & 0
\end{matrix}
\right]
=
\left[
\begin{matrix}
I_n  & 0\\
0 & 0 
\end{matrix}
\right].
$$
Hence $\coker(U U^t)$ is also free.

For the reverse implication, assume $\coker(U U^t) \cong R^{m-n}$.
This implies that $\rank_K U U^t = n$.  Note that $\im(U U^t) \subseteq \im U$,
which has two implications.  Firstly, it implies 
$$
(n=)\rank_K U U^t \leq \rank_K U
$$ 
and since $U$ has dimensions $m \times n$, this
forces 
$$
\rank_K U=n=\rank_K U U^t.
$$

Secondly, it gives a surjection
$$
\coker(U U^t) \overset{\pi}{\twoheadrightarrow} \coker(U)
$$
between two $R$-modules of the same rank.  Since this implies that
$\ker \pi$ has $K$-dimension $0$, it must be all $R$-torsion as an $R$-module, and hence zero
because it is a submodule of the free module $\coker(U U^t) \cong R^{m-n}$.
Thus $\pi$ is an isomorphism, and $\coker(U) \cong R^{m-n}$.
\end{proof}

\subsection{Parametrized Smith forms}

Recall that a square matrix $A$ in $R^{n \times n}$ is said to have a parametrized
Smith form if $A+tI$ in $R[t]^{n \times n}$ has a Smith form over $R[t]$.
This has strong consequences.

\begin{proposition}
\label{prop:smith-eigenvalue-relation}
Let $R$ be an integral domain, with field of fractions $K$, and
$\Kbar$ the algebraic closure of $K$.

Assume $A$ in $R^{n \times n}$ has a parametrized
Smith form for $A+tI$ with the diagonal entries $s_1(t),\ldots,s_n(t)$ in $R[t]$.

\begin{enumerate}
\item[(i)]
The $s_i(t)$ can all chosen with leading coefficient $1$ (monic) in $R[t]$.
\item[(ii)]
For each $\lambda$ in $\Kbar$, the (geometric) eigenvalue multiplicity
$$
\dim_\Kbar \ker(A-\lambda I)=m
$$ 
can be read off of the $s_i(t)$ as follows: $t+\lambda$ is a
factor in $\Kbar[t]$ of the last $m$ entries 
$$
s_n(t),s_{n-1}(t),\ldots,s_{n-(m-1)(t)}
$$
but no other $s_i(t)$.
\item[(iii)]
$A$ is semisimple (=diagonalizable in $\Kbar^{n \times n}$) if and only
the $s_i(t)$ have no repeated factors when split competely in $\Kbar[t]$.
\item[(iv)] 
In the situation where $A$ is semisimple, the $s_i(t)$ factor 
as follows:
$$
s_{n+1-i}(t) = \prod_{\substack{\lambda \in \Kbar: 
                 \\ \dim_\Kbar \ker ( A -\lambda I ) \geq i }} ( t + \lambda).
$$
So, for example, the last parametrized Smith entry factors as
$$
s_n(t) = \prod_{\text{ eigenvalues } \lambda \in \Kbar \text{ for } A } (t+\lambda).
$$
\end{enumerate}
\end{proposition}

\begin{proof}
  By assumption, there exist $P(t), Q(t)$ lying in $GL_n(R[t])$ for which
\begin{equation}
\label{eqn:parametrized-smith}
P(t) (A+tI) Q(t) = \diag( s_1(t),\ldots,s_n(t) )
\end{equation}
where here $\diag(a_1,\ldots,a_n)$ denotes the $n \times n$ square diagonal matrix
with eigenvalues $a_1,\ldots,a_n$.
Since $R$ is an integral domain, the units of $R[t]$ are exactly the
units $R^\times$, that is, independent of $t$, and hence 
$\det P(t), \det Q(t)$ both lie in $R^\times$.  
Thus one has
\begin{equation}
\label{det-relation}
\det P(t) \det(A+tI) \det Q(t) = s_1(t) \cdots s_n(t)
\end{equation}
where $\det P(t), \det Q(t)$ lie in $R^\times$.

Assertion (i) now follows:  up to scalars in $R^\times$,
the leading terms of the $s_i(t)$ multiply to give
the leading term of $\det(A+tI)$, which is $1$.  Hence these leading terms are
all in $R^\times$, and can be scaled away.

For assertion (ii),
note that for any 
$\lambda$ in $\Kbar$, one can set $t=-\lambda$ in \eqref{eqn:parametrized-smith} 
to obtain
\begin{equation}
\label{eqn:specialized-parameter}
P(-\lambda) (A-\lambda I) Q(-\lambda) = \diag( s_1(-\lambda),\ldots,s_n(-\lambda) )
\end{equation}
where both $P(-\lambda), Q(-\lambda)$ lie in $GL_n(\Kbar)$ since their determinants
are independent of $\lambda$ and lie in $R^\times (\subset K^\times \subset \Kbar^\times$).
Consequently, one has
$$
m=\dim_\Kbar \ker ( A -\lambda I )=\dim_\Kbar \ker \diag( s_1(-\lambda),\ldots,s_n(-\lambda) )
$$
if and only if $t=-\lambda$ is a root of exactly $m$ of the polynomials $s_i(t)$,
that is, if and only if $t+\lambda$ is a factor in $\Kbar[t]$ of 
$s_n(t),s_{n-1}(t),\ldots,s_{n-(m-1)}(t)$ and no others.

For assertions (iii) and (iv),
recall that one has $A$ semisimple if and only
if for each $\lambda$ the (geometric) multiplicity
$m=\dim_\Kbar \ker ( A -\lambda I )$
is the same as the multiplicity of $t=-\lambda$ as a root of
$\det(A+tI)$.  We claim that the latter is the same as its
multiplicity as a root of $s_1(t) \cdots s_n(t)$, due to \eqref{det-relation}.
Since we know $t=-\lambda$ is always a root of the last $m$ factors
in this product and no others, it is a root of multiplicity $m$
if and only if it is a simple root for each of these last $m$ factors.
This gives all the remaining assertions. 
\end{proof}

We illustrate, by way of examples,
that having $A$ in $R^{n \times n}$ with a parametrized Smith form
for $A+tI$ in $R[t]$ implies neither that $A$ is semisimple, nor that its
eigenvalues lie in $R$ or $K$.

\begin{example}
Let $R=\ZZ$ and 
$$
A=
\left[
\begin{matrix}
0 & 1 \\
0 & 0
\end{matrix}
\right]
$$
which is not semisimple.
Nevertheless, one can bring $A+tI$ into Smith form over $\ZZ[t]$ by these row and column operations:
$$
A+tI =
\left[
\begin{matrix}
t & 1 \\
0 & t
\end{matrix}
\right]
\sim
\left[
\begin{matrix}
0 & 1 \\
-t^2 & t
\end{matrix}
\right]
\sim
\left[
\begin{matrix}
0 & 1 \\
-t^2 & 0
\end{matrix}
\right]
\sim
\left[
\begin{matrix}
1 & 0 \\
0 & t^2
\end{matrix}
\right].
$$
Note that one of the diagonal entries $s_2(t)=t^2$ has a double root in
this case, as predicted by Proposition~\ref{prop:smith-eigenvalue-relation} above.
\end{example}

\begin{example}
Let $R=K=\QQ$ and 
$$
A=
\left[
\begin{matrix}
0 & 1 \\
-1 & 0
\end{matrix}
\right]
$$
which is semisimple, but with eigenvalues $\pm i$ that do not lie in $R=K$.
However, since $R$ is a field and hence $R[t]$ is a PID, one can certainly bring
$A+tI$ into Smith form over $\QQ[t]$ by these row and column operations:
$$
A+tI=
\left[
\begin{matrix}
t & 1 \\
-1 & t
\end{matrix}
\right]
\sim
\left[
\begin{matrix}
0 & t^2+1 \\
-1 & t
\end{matrix}
\right]
\sim
\left[
\begin{matrix}
0 & t^2+1 \\
-1 & 0
\end{matrix}
\right]
\sim
\left[
\begin{matrix}
1 & 0 \\
0 & t^2+1
\end{matrix}
\right].
$$
Note that each of the eigenvalues is at most a simple root of
the last Smith entry $s_2(t)=t^2+1=(t+i)(t-i)$.
\end{example}

\begin{remark}
\label{fraction-field-and-PID-consequences}
  Note that if $R$ is itself a $PID$, then for every value $k$ in $R$, the
matrix $A+kI$ in $R^{n \times n}$ will have a Smith form over $R$.  Thus whenever
$A+tI$ also has a Smith form in $R[t]$, it will predict these
Smith forms over $R$ by setting $t=k$.  

  Note also that for any integral domain $R$, since its field of fractions $K$ always has
$K[t]$ a PID, a Smith form for $A+tI$ in $K[t]$ is guaranteed to exist.  Thus whenever
$A+tI$ also has a Smith form in $R[t]$, it will predict the Smith form in $K[t]$.
However, knowledge of the eigenvalues of $A$ in $\Kbar$ will already predict this Smith form 
for $A+tI$ in $K[t]$, by the argument in Proposition~\ref{prop:smith-eigenvalue-relation}.
\end{remark}

\begin{remark}
  To give some sense of how rarely one finds a parametrized Smith form for $A+tI$ 
when $R$ is not a field and $A$ lies in $R^{n \times n}$, we note its scarcity among
among {\it diagonal matrices}.

\begin{proposition}
\label{prop:diagonal-scarcity}
Let $A=\diag(\lambda_1,\ldots,\lambda_n)$ be a diagonal matrix in $R^{n \times n}$.
If $A+tI$ has a Smith form in $R[t]$ then every nonzero difference $\lambda_i - \lambda_j$
is a unit in $R^\times$.

In particular, for $R=\ZZ$, a diagonal matrix $A$ has a parametrized Smith form
for $A+tI$ in $\ZZ[t]$ if and only if the diagonal entries take on at most two values in
$\ZZ$, and those two values differ by $1$.
\end{proposition}
\begin{proof}
Let $\lambda$ be among the diagonal entries $\{\lambda_1,\ldots,\lambda_n\}$.
We will show that $\lambda_i - \lambda$ lies in $R^\times$ for all $i$ using
downward induction on the multiplicity $m$ with which $\lambda$ occurs in the list.
Reindex so that $\lambda$ occurs as the last $m$ diagonal entries in $A$.
Let $P(t) (A+tI) Q(t) = \diag(s_1(t),\ldots,s_n(t))$, where we know from
Proposition~\ref{prop:smith-eigenvalue-relation} how to factor $s_i(t)$ over $R$.
Setting $t=-\lambda$ one obtains
$$
\begin{aligned}
&P(-\lambda) \diag( \lambda_1 - \lambda,\ldots, \lambda_{n-m} - \lambda, 
               \underbrace{0,0,\cdots,0}_{m\text{ times }}) Q(-\lambda)\\
&\quad =
\diag(s_1( - \lambda),\ldots, s_{n-m}( - \lambda), 
               \underbrace{0,0,\cdots,0}_{m\text{ times }}). 
\end{aligned}
$$
However the exact form of each of the nonzero terms $s_i(-\lambda)$ on the
right shows that it is a product of factors of the form $\lambda_j - \lambda$
where $\lambda_j$ has larger multiplicity than $m$.  Thus by induction,
each such factor $\lambda_j - \lambda$ is a unit, and hence so is $s_i(-\lambda)$.

Taking cokernels of both sides now gives an $R$-module isomorphism
$$
R^m \oplus \bigoplus_{i=1}^{n-m} R/(\lambda_i - \lambda) \cong R^m.
$$
This forces each $R/(\lambda_i - \lambda)=0$ (and hence $\lambda_i - \lambda$ is
forced to be a unit): any $x$ in $R/(\lambda_i - \lambda)$ gives rise to an element
$(x,0,0,\ldots)$ killed by $\lambda_i - \lambda$ on the left,
but since only $0$ is killed by it in $R^m$ on the right, this forces $x=0$.

The assertion for $R=\ZZ$ then follows easily from the fact that $\ZZ^\times=\{\pm 1\}$,
and the difference $+1-(-1)=2$ is {\it not} in $\ZZ^\times$.
\end{proof}

\end{remark}

\begin{example}
The following examples illustrate some of the previous remarks.
$$
A=
\left[
\begin{matrix}
-1 & 0 \\
0 & 0 
\end{matrix}
\right]
$$
in $\ZZ^{2 \times 2}$ has a parametrized Smith form, as
$$
A+tI=
\left[
\begin{matrix}
t-1 & 0 \\
0 & t 
\end{matrix}
\right]
\sim
\left[
\begin{matrix}
t-1 & 0 \\
t & t 
\end{matrix}
\right]
\sim
\left[
\begin{matrix}
-1 & -t \\
t & t 
\end{matrix}
\right]
\sim
\left[
\begin{matrix}
-1 & -t \\
 0& t-t^2 
\end{matrix}
\right]
\sim
\left[
\begin{matrix}
 1 & 0 \\
 0& t(t-1) 
\end{matrix}
\right]
$$
Consequently, it has the same Smith form over $\QQ[t]$, and
for any integer $k$, the matrix $A+kI$ will have Smith form over $\ZZ$
given by setting $t=k$.  In particular, the cokernel $\ZZ^2/\im(A+kI)$
is always isomorphic to the cyclic group  $\ZZ/k(k-1)\ZZ$, interpreting this
as $\ZZ/0\ZZ = \ZZ$ when $k = 0,1$.

On the other hand, the matrix
$$
A=
\left[
\begin{matrix}
-1 & 0 \\
0 & 1 
\end{matrix}
\right]
$$
has {\it no} Smith form in $\ZZ[t]$ for 
$$
A+tI=
\left[
\begin{matrix}
t-1 & 0 \\
0 & t+1 
\end{matrix}
\right]
$$
according to Proposition~\ref{prop:diagonal-scarcity}.
The Smith form for $A+tI$ in $\QQ[t]$ has diagonal entries $1, (t-1)(t+1)$, but
requires dividing by $2$ in order to put it into this form.
One also sees a wider range of behaviors for the cokernel of $A+kI$ over $\ZZ$
as the integer $k$ varies:
$$
\coker(A+kI) \cong
\begin{cases}
\ZZ/(k-1)\ZZ \oplus \ZZ/(k+1)\ZZ &\text{ if }k\text{ is odd},\\
\ZZ/(k-1)(k+1)\ZZ & \text{ if }k\text{ is even}.
\end{cases}
$$
\end{example}

%%%%%%%%%%%%%%%%%%%%%%%%%%%%%%%%%%%%%%%%%%%%%%%%%%%%%%%%%%%%%%%%%%


\begin{thebibliography}{}
\bibitem{BergeronLamLi}
N. Bergeron, T. Lam, and H. Li
Combinatorial Hopf algebras and Towers of Algebras,
{\tt arXiv:0710.3744}.

\bibitem{CP} 
V. Chari and A. Pressley, 
\textit{A Guide to Quantum Groups},
Cambridge University Press, 1994.
 
\bibitem{DF}
D.S. Dummit and R.M. Foote, 
\textit{Abstract Algebra}, 
Prentice Hall, Englewood Cliffs, NJ, 1991.

\bibitem{Fomin1}
S. Fomin, 
\textit{Duality of graded graphs}, 
J. of Alg. Combin., {\bf 3} (1994), 357-404.

\bibitem{Fomin2}
S. Fomin, 
\textit{Schensted algorithms for dual graded graphs}, 
J. of Alg. Combin., {\bf 4} (1995), 5-45.

\bibitem{F}
W. Fulton, 
\textit{Young Tableaux}, 
Cambridge Univ. Press, London, 1997.

\bibitem{FH}
W. Fulton and J. Harris, 
\textit{Representation Theory, a First Course}, 
Springer GTM 129, 1991. 

\bibitem{GHJ}
F. M. Goodman, P. de la Harpe, and V. F. R. Jones, 
\textit{Coxeter Graphs and Towers of Algebras}, 
Math. Sci. Res. Inst. Publ., no. 14, Springer-Verlag, New York, 1989.

\bibitem{Kuperberg}
G. Kuperberg,
\textit{Kasteleyn cokernels},
Electron. J. Combin. 9 (2002), 30pp.


\bibitem{L-R}
J. L. Loday and M. O. Ronco,
\textit{Hopf algebra and the planar binary trees},
Adv. Math., {\bf 139} (1998), no. 2, 293-309. 

\bibitem{M}
I. G. Macdonald, 
\textit{Symmetric Functions and Hall Polynomials}, 
Oxford Univ. Press, 1979.

\bibitem{Mi}
A. Miller,
\textit{Smith invariants of $UD$ and $DU$ linear operators in differential posets},
REU report.

\bibitem{Ne}
M. Newman,
\textit{Integral Matrices},
Vol. 45 in Pure and App. Math. Academic press, 
N.Y. and London, 1972.

\bibitem{Nu}
Y. Numata, 
\textit{Pieri's formula for generalized Schur polynomials}, 
FPSAC, 2006.

\bibitem{Nz}
J. Nzeutchap, 
\textit{Dual graded graphs and Fomin's $r$-correspondences associated to the Hopf 
algebras of planar binary trees, quasi-symmetric functions and noncommutative symmetric functions}, 
FPSAC, 2006.

\bibitem{Nz2}
J. Nzeutchap,
\textit{Binary search tree insertion, the hypoplactic insertion, and dual graded graphs},
{\tt  arXiv:0705.2689}.


\bibitem{Okada}
S. Okada, 
\textit{Wreath products by the symmetric groups and product posets of Young's lattices}, 
J. Combin. Theory Ser. A, \textbf{55} (1990), 14--32.

\bibitem{O}
S. Okada, 
\textit{Algebras associated to the Young-Fibonacci lattice}, 
Trans. Amer. Math. Soc., \textbf{346} (1994), no. 2, 549-568.

\bibitem{S1}
R. P. Stanley, 
\textit{Differential posets}, 
J. Amer. Math. Soc., \textbf{1} (1988), 919-961.

\bibitem{S2}
R. P. Stanley, 
\textit{Variations on differential posets}, Invariant Theory and Tableaux (D. Stanton, ed.), 
IMA Vol. Math. Appl., no. 19, Springer, New York, 1988, pp. 145-165.

\bibitem{S3}
R. P. Stanley, 
\textit{Enumerative Combinatorics, Volume 1}, 
Cambridge University Press, 1997.

\bibitem{S4}
R. P. Stanley, 
\textit{Enumerative Combinatorics, Volume 2}, 
Cambridge University Press, 1999.

\bibitem{Stembridge}
J. Stembridge,
\textit{On the eigenvalues of representations of reflection groups and wreath products}
Pac. J. Math., \textbf{140} (1989), 353--396.

\bibitem{W}
R. M. Wilson, 
\textit{A diagonal form for the incidence matrices of $t$-subsets vs. $k$-subsets}, 
Europ. J. Comb., \textbf{11} (1990), 609-615.
\end{thebibliography}
\end{document}